\documentclass{amsart}

\usepackage{hyperref}
\usepackage{amsfonts,amsmath,amssymb,amsthm,amscd}
\usepackage{graphicx}
\usepackage{tikz}
\usepackage{pgfplots}
\usepackage{todonotes}

\newcounter{mainTheorem}
\newtheorem{theorem}{Theorem}[section]

\newtheorem{maintheorem}[mainTheorem]{Theorem}
\newtheorem{lemma}[theorem]{Lemma}
\newtheorem{proposition}[theorem]{Proposition}
\newtheorem{corollary}[theorem]{Corollary} 
\newtheorem{definition}[theorem]{Definition}
\newtheorem{remark}[theorem]{Remark}

\newtheorem{example}[theorem]{Example}

\newcommand\Acal{{\mathcal A}}
\newcommand\Bcal{{\mathcal B}}
\newcommand\Ccal{{\mathcal C}}
\newcommand\Fcal{{\mathcal F}}

\newcommand\Ical{{\mathcal I}}

\newcommand\Kcal{{\mathcal K}}
\newcommand\Qcal{{\mathcal Q}}
\newcommand\Gcal{{\mathcal G}}
\newcommand\Lcal{{\mathcal L}}
\newcommand\Pcal{{\mathcal P}}
\newcommand\Tcal{{\mathcal T}}
\newcommand\Scal{{\mathcal S}}
\newcommand\dd{{\bf d}}
\newcommand\ee{{\bf e}}
\newcommand\ff{{\bf f}}
\newcommand\uu{{\bf u}}
\newcommand\vv{{\bf v}}
\newcommand\RR{{\mathbb R}}
\newcommand\CC{{\mathbb C}}
\newcommand\ZZ{{\mathbb Z}}
\newcommand\Ztwo{{\mathbb Z}/2{\mathbb Z}}
\newcommand\FF{{\mathbb F}}

\newcommand\Pfrak{{\mathfrak P}}

\newcommand\qbv{{q}}
\newcommand\abv{{a}}
\newcommand\pj{{\rm proj}}
\DeclareMathOperator{\Image}{Im}

\DeclareMathOperator{\rank}{rank}
\DeclareMathOperator{\supp}{supp}
\DeclareMathOperator{\Hom}{Hom}
\DeclareMathOperator{\Sal}{Sal}
\DeclareMathOperator{\OS}{OS}
\DeclareMathOperator{\BC}{BC}

\newcommand{\comment}[1]{}

\begin{document}
\title{Filtrations of tope spaces of oriented matroids}
\author{Kris Shaw and Chi Ho Yuen} 

\maketitle

\begin{abstract}
We compare three filtrations of the tope space of an oriented matroid. 
The first is the dual Varchenko--Gelfand degree filtration, the second filtration is from Kalinin's spectral sequence, and the last one derives from Quillen's augmentation filtration. We show that all three filtrations and the respective maps coincide over $\mathbb{Z}/ 2\mathbb{Z}$. 

We also show that the dual Varchenko--Gelfand degree filtration can be made into a filtration of the $\mathbb{Z}$-sign cosheaf on the fan of the underlying matroid.
This was previously carried out with $\mathbb{Z}/ 2\mathbb{Z}$-coefficients by the first author and Renaudineau using the Quillen filtration and has applications to real algebraic geometry via patchworking.
\end{abstract}

\tableofcontents

\section{Introduction}

The {\em tope space} of a real hyperplane arrangement $\Acal$ is the homology of its complement. While the complement simply consists of disjoint polyhedra, the interaction with the combinatorics and (complexified) geometry of $\Acal$ leads to interesting filtrations of the tope space.
By the Folkman--Lawrence Topological Representation Theorem \cite{FolkmanLawrence}, every oriented matroid can be realised by a {\em pseudohyperplane arrangement}, and the above notions can be generalised to the setting of {\em oriented matroids}.
The topes of an oriented matroid are in correspondence with the connected components of the complement of the pseudoarrangement, and the role of the complexification of the complement is played by the {\em Salvetti complex}.

In this paper, we start with studying the $\mathbb{Z}/2 \mathbb{Z}$-vector space generated by the topes of an oriented matroid, which we call the {\em $\mathbb{Z}/2 \mathbb{Z}$-tope space}. 
In the first part of this paper we show the equality of three filtrations of the $\Ztwo$-tope space of an oriented matroid in the literature. 
The first filtration is the mod $2$ reduction of the dual of the {\em Varchenko--Gelfand degree filtration} $\Pcal_\bullet$ on the $\ZZ$-tope space \cite{GV87}. This filtration comes from the ring of Heaviside functions, and we denote its mod $2$ reduction by $\overline{\Pcal}_\bullet$. The second filtration is the {\em Kalinin filtration} $\Kcal_\bullet$ \cite{Kalinin}, which is induced by a spectral sequence on a topological space with involution. The third filtration, denoted $\mathcal{Q}_{\bullet}$, is derived from {\em Quillen's augmentation filtration} on initial matroids; it was introduced by Renaudineau and the first author in \cite{RS}, and further developed by Rau, Renaudineau, and the first author in \cite{RRS2}.
We briefly state our first main result before describing the filtrations in more detail.

\begin{maintheorem} \label{mainthm:Z2}
Let $M$ be an oriented matroid.
The filtrations of the $\mathbb{Z}/2\mathbb{Z}$--tope space $\Qcal_{\bullet}(M), \Kcal_{\bullet}(M), $ and $\overline{\Pcal}_{\bullet}(M)$ all coincide.
\end{maintheorem}

The Salvetti complex $\Sal_M$ of an oriented matroid \cite{Sal87} is a CW complex playing the role of the complement of a complexified hyperplane arrangement. In fact, when $M$ is the oriented matroid of a real hyperplane arrangement, the Salvetti complex is a deformation retract of (hence homotopic equivalent to) the complexification of the complement. For any oriented matroid $M$, there is an involution on $\Sal_M$ and the fixed locus under this involution consists of a discrete collection of points corresponding to the topes of the oriented matroid.

The Varchenko--Gelfand filtration arises by recognising that the ring of $\ZZ$-valued functions over the $\ZZ$-tope space can be expressed as a quotient of the polynomial ring generated by {\em Heaviside functions} (indicator functions of halfspaces). The degree filtration of the polynomial ring induces a filtration $\Pcal^\bullet$ of said ring of functions, which in turn gives a filtration of the tope space itself by taking duals (or orthogonal complements).
Varchenko proved that the dual degree filtration is equal to the {\em asymptotic filtration} of a real hyperplane arrangement \cite{Var93}, which has applications in Lie theory and singularity theory; see Remark~\ref{rem:asym} for its definition.
Working from the perspective of singularity theory, Denham \cite{Denham_OS} defined an isomorphism $\abv_p: \Pcal_p(M)/ \Pcal_{p+1}(M) \to H_p(\Sal_M; \ZZ)$ for every $p$.
We can take the mod $2$ reduction of everything to obtain a filtration $\overline{\Pcal}_\bullet(M)$ of the $\Ztwo$-tope space and isomorphisms $\overline{\abv}_p :
\overline{\Pcal}_p(M) / \overline{\Pcal}_{p+1}(M) \to H_p(\Sal_M; \Ztwo)$. See also the work of Proudfoot on the relation between the Varchenko--Gelfand rings and Salvetti complexes in an equivariant setting \cite{Pro06}.

The second filtration is the Kalinin filtration $\Kcal_\bullet$. 
The Kalinin filtration of a real algebraic variety, or more generally, a topological space (that is reasonably nice, e.g., a CW complex) with an involution representing complex conjugation, is induced by a spectral sequence that starts with the homology of the whole space over $\Ztwo$ and converges to the homology of the real part of the variety (respectively, the fixed locus of the space).
This spectral sequence can be thought of as a categorification of the Smith--Thom inequality.
As mentioned above, the fixed locus of the involution on $\Sal_M$ is a collection of points corresponding to the topes of $M$. Hence we obtain a filtration of the tope space together with an isomorphism $bv_p : \Kcal_p(M)/\Kcal_{p+1}(M) \to H_p(\Sal_M;\ZZ/2\ZZ)$ for every $p$ from the spectral sequence.

The last filtration derives from Quillen's augmentation filtration \cite{Quillen}, which was originally used in \cite{RS} in the case of orientations of the uniform matroid $U_{n, n+1}$. 
The usual Quillen's augmentation filtration is defined on $(\mathbb{Z}/2\mathbb{Z})[V]$, the group algebra of a $\Ztwo$-vector space $V$.
For a general oriented matroid $M$, denote its underlying matroid by $\underline{M}$. The collection of topes of $M$ {\em adjacent} to a complete flag of flats of $\underline{M}$ has the structure of an affine subspace of $(\mathbb{Z}/2\mathbb{Z})^E$. By choosing an arbitrary origin of this affine space, {we obtain a filtration of its group algebra which is independent of the choice made.
Taking the sum of these filtrations over all complete flags yields a filtration $\Qcal_\bullet$ of the tope space of $M$.
Moreover, \cite{RS} (implicitly) described an isomorphism $\qbv_p$ between $\Qcal_p/\Qcal_{p+1}$ and the $p$-th graded piece $\OS_p(\underline{M};\Ztwo)$ of the dual of the {\em Orlik--Solomon algebra} of $\underline{M}$.

Once we have identified the filtrations of the $\ZZ/2\ZZ$--tope space in Theorem \ref{mainthm:Z2}, we also compare their associated maps. In order to do that, we make use of the (dual of the) isomorphism given by Bj\"{o}rner--Ziegler \cite{BZ92}, denoted by $BZ^\vee$, between $H_p(\Sal_M;\ZZ/2\ZZ)$ and $ \OS_p(\underline{M};\ZZ/2\ZZ)$.

\begin{maintheorem} \label{mainthm:maps}
 As maps from $\Kcal_p(M)=\overline{\Pcal}_p(M)$ to $H_p(\Sal_M;\Ztwo)$, $bv_p,\overline{\abv}_p$  are equal. Furthermore, we have the following commutative diagram
\begin{equation}
\begin{CD}
\overline{\Pcal}_p(M)    @>\overline{\abv}_p>>  H_p(\Sal_M;\Ztwo)\\
@V\parallel VV        @V BZ^\vee VV\\
\Qcal_p(M)      @>\qbv_p>>  \OS_p(\underline{M};\Ztwo)
\end{CD}
\end{equation}
In particular, the three maps descend to the same isomorphism from the $p$-th associated gradeds of the filtrations to $H_p(\Sal_M;\Ztwo)$ up to the isomorphism $BZ^\vee$ as explained above.
\end{maintheorem}

Quillen's filtration applied to $\mathbb{Z}/2\mathbb{Z}$-tope spaces was used by Renaudineau and the first author to obtain a combinatorial spectral sequence converging to the homology of the real part of hypersurfaces arising from Viro's patchworking. This patchworking construction has been generalised to tropical manifolds equipped with real phase structures in \cite{RRS2}. The result of the patchworking is the \emph{real part} of the tropical manifold and a {\em sign cosheaf} on the tropical manifold which computes the homology of this real part; the sign cosheaf of a tropical space equipped with a real phase structure is a cosheaf whose stalks are the tope spaces of oriented matroids.

The above two theorems further suggest that the spectral sequence from \cite{RS}, and more generally \cite{RRS2}, can be thought of as a tropical/combinatorial version of Kalinin's filtration. A similar spectral sequence for real semi-stable degenerations satisfying some conditions was recently found by Ambrosi and Manzaroli using equivariant cohomology and log geometry \cite{ambrosimanzaroli}.

The application of Quillen's filtration to the $\ZZ$-tope space does not stabilise, see Example~\ref{ex:Quillen_no_Z}. 
Moreover, a naive extension of the Kalinin filtration to the $\ZZ$-tope space is not well-behaved, see Remark~\ref{rm:Kalinin}. 
However, the dual Varchenko--Gelfand filtration is defined on the $\ZZ$-tope space of an oriented matroid.
Here we show that this filtration can be turned into a filtration of the $\mathbb{Z}$-sign cosheaf on the polyhedral fan of a matroid. To do this, we first establish the functoriality of the filtration with respect to taking initial matroids. In the following statement, $\Acal_p$ is the dual of the {\em Cordovil algebra} and $\tilde{\abv}_p$ is the map introduced in Proposition~\ref{prop:gamma_epsilon}; see Section~\ref{sec:OS_Cord} and \ref{sec:Cordovil} for details.

\begin{proposition}\label{prop:Z-functorial}
The dual Varchenko--Gelfand filtration is functorial with respect to initial matroids. More precisely, for any pair of flags $\Fcal \subset \Fcal'$ we have $\Pcal_p(M_{\Fcal'})\subset \Pcal_p(M_\Fcal)$ and the following diagram commutes 
\begin{equation} \label{eq:GV_sheaf}
\begin{CD}
\Pcal_p(M_{\Fcal'})     @>\tilde{\abv}_p>> \Acal_p(M_{\Fcal'};\ZZ)\\
@V{\iota}VV        @V{\iota}VV\\
\Pcal_p(M_\Fcal)     @>\tilde{\abv}_p>>   \Acal_p(M_{\Fcal};\ZZ).
\end{CD}
\end{equation}
\end{proposition}

The fan of a matroid has cones corresponding to the flags of flats of the matroid, hence to its initial matroids. 
The $\ZZ$-sign cosheaf on a matroid fan of an oriented matroid associates to each face the $\ZZ$-tope space of the initial matroid corresponding to the face. 
The cosheaf maps are inclusions of tope spaces.

Using the above proposition, we can then establish the following theorem for the $\ZZ$-sign cosheaf of an oriented matroid, which is an extension of the $\Ztwo$-sign cosheaf from \cite{RRS2}. We let $\Pfrak_p$ denote a cosheaf formed by the dual Varchenko-Gelfand filtration.   Notice that in the theorem below, we  replace the Orlik--Solomon algebra with the {\em Cordovil algebra}. 

\begin{maintheorem} \label{mainthm:Z_cosheaf}
The integral dual Varchenko--Gelfand filtration produces a filtration of the $\mathbb{Z}$-sign cosheaf on the fan (of the underlying matroid) of an oriented matroid,
$$\Pfrak_d  \subset  \dots \subset \Pfrak_{p+1} \subset \Pfrak_p \subset \dots \subset \Pfrak_1 \subset \mathcal{S}.$$
Moreover, for every $p$ there is a short  exact sequence of cosheaves 
$$0 \to \Pfrak_{p+1} \to \Pfrak_p \to \mathfrak{A}_p \to 0,$$ where $\mathfrak{A}_p$ is the $p$-th {\em Cordovil cosheaf}. 

\end{maintheorem}

Via the spectral sequence of the above filtration, the Cordovil cosheaves on a tropical manifold with a real phase structure provide information on the $\ZZ$-homology groups of the real part. In future work, it will be interesting to compare the homology of these Cordovil cosheaves with the {\em $\ZZ$-tropical homology cosheaves} \cite{IKMZ}.

\subsection*{Acknowledgement}
Both authors were supported by the Trond Mohn Foundation project ``Algebraic and Topological Cycles in Complex and Tropical Geometries'', and acknowledge the support of the Centre for Advanced Study (CAS) in Oslo, Norway, which funded and hosted the Young CAS research project ``Real Structures in Discrete, Algebraic, Symplectic, and Tropical Geometries'' during the 2021/2022 and 2022/2023 academic years.

Both authors thank Nick Proudfoot for pointing out Cordovil algebra as the better-behaved object for Theorem~\ref{mainthm:Z_cosheaf}, and Graham Denham for providing helpful references; they are also grateful to the anonymous referee whose comments and suggestions helped us to improve this paper. 
CHY was also supported by the Danish National Research Foundation project DNRF151 and the Ministry of Science and Technology of Taiwan project MOST 113-2115-M-A49-004-MY2 during his affiliation to the University of Copenhagen and National Yang Ming Chiao Tung University, respectively; he would like to thank Henry Tsang for the numerous conversations, Galen Dorpalen-Barry for the discussion on this project, and Basile Coron for the explanation on connections with operad theory.

\section{Background}

\subsection{Oriented Matroids}
We follow the conventions for oriented matroids from \cite{BLSWZ} and we mostly use the covector description of oriented matroids \cite[Section 4]{BLSWZ}.
 The {\em covectors} of an oriented matroid $M$ are elements of $\{0, +, -\}^E$ where $E$ is the ground set of $M$. See \cite[Section 4.1]{BLSWZ} for the set of axioms for covectors. 
The partial ordering of the covectors of $M$ is induced via the relation $0 < \pm$, and $+$ and $-$ are incomparable. 
The covector lattice of an oriented matroid $M$ is denoted by $\mathcal{L}(M)$, where we drop the $M$ when the oriented matroid is clear from context.

When an oriented matroid is realised by a real hyperplane arrangement consisting of hyperplanes indexed by $E$, the covectors are precisely the sign vectors encoding the relative position of the points in the ambient space (see Example~\ref{ex:CovectorsU23}): we fix a positive halfspace for the $i$-th hyperplane, and record which side a point is in (including the case of being on the hyperplane itself) by the $i$-th coordinate of the covector.
The hyperplanes provide a cellular decomposition of the ambient space, and the collection of covectors index the cells.
By the Folkman--Lawrence Topological Representation Theorem \cite{FolkmanLawrence}, every oriented matroid arises from a {\em pseudohyperplane arrangement}, which allows us to extend the above intuition to all oriented matroids.
The topes of an oriented matroid are in correspondence with the connected components of the complement of the pseudoarrangement and all points of the ambient space can be assigned a covector upon choosing a positive side of each pseudohyperplane.

\begin{example}\label{ex:CovectorsU23} \rm
The real hyperplane arrangement in Figure~\ref{diag:U23} consists of three hyperplanes $\ell_1,\ell_2,\ell_3$. 
For each hyperplane, the positive halfspace is pointed out by the arrows in the figure. 
The arrangement defines an oriented matroid with 13 covectors as labeled by letters. Some examples as sign vectors are $T_1=+++, T_3=--+, \zeta_2=-0+,\zeta_4=0--, O=000$. The underlying matroid is $U_{2,3}.$
\end{example}

\begin{figure}[]
\begin{center}
    \includegraphics[scale=0.35]{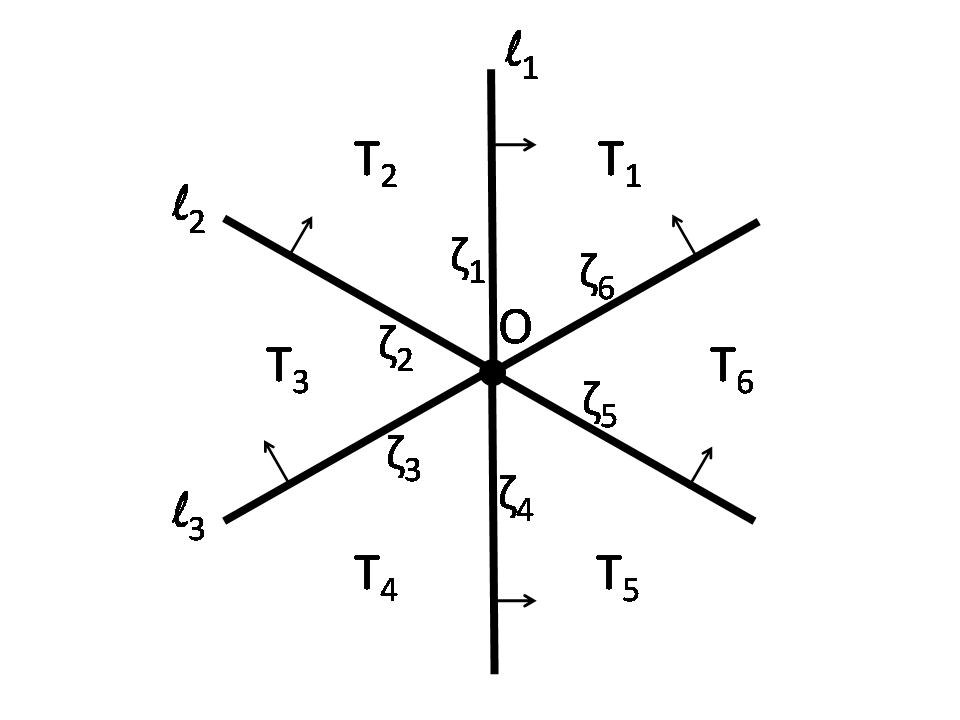}
  \label{diag:U23}
  \end{center}
  \vspace{-0.5cm}
  \caption{A real hyperplane arrangement realising $U_{2,3}$.}
\end{figure}

Every oriented matroid $M$ has an underlying matroid $\underline{M}$. The \emph{flats}  of $\underline{M}$  are obtained from the covectors of $M$ by considering $E \backslash \supp(L)$ over all covectors $L \in \mathcal{L}(M)$, here $\supp(L)\subset E$ is the support of the vector. We will sometimes speak of flats of the oriented matroid $M$, meaning flats of $\underline{M}$.

We also make use of the rank function of the underlying matroid which we denote by $\rank \colon 2^E \to \mathbb{Z}_{\geq 0}$.
A \emph{circuit} of $\underline{M}$ is a subset $C\subset E$ such that $\rank(C\setminus i)=\rank(C)=|C|-1$ for any $i\in A$. An oriented matroid $M$ can as well be described by its collections of {\em signed circuits}, which are elements of $\{+,-,0\}^E$ supported on circuits of $\underline{M}$ \cite[Section~3.2]{BLSWZ}. A \emph{loop} of $\underline{M}$ (hence $M$) is an element $i\in E$ such that $\rank(\{i\}) = 0$.
Throughout we assume that the oriented matroid $M$ is loopless. 

There are operations of deletion and contraction on oriented matroids.
If $M$ is the oriented matroid on ground set $E$ obtained from an pseudohyperplane arrangement then for $i \in E$, the deletion  $M \backslash i$ is the oriented matroid arising from removing the $i$-th pseudohyperplane from the arrangement. The contraction $M / i$ is the oriented matroid obtained by restricting the pseudohyperplane arrangement to the $i$-th sphere. The deletion and contraction operations can be defined for arbitrary subsets of $E$. Moreover, the restriction of $M$ to $I \subset E$ is the oriented matroid $M|_I = M \backslash (E \backslash I)$. The operations of deletion and contraction commute which allow us to define initial oriented matroids. 

\begin{definition}
Let $\Fcal:\emptyset=F_0 \subsetneq \ldots \subsetneq  F_t=E$ be a flag of flats of a loopless oriented matroid $M$.
The {\em initial oriented matroid} $M_{\Fcal}$ with respect to the flag is the (necessarily loopless) oriented matroid 
\begin{equation} \label{def:IM}
    \bigoplus_{i=0}^{t-1} M|_{F_{i+1}}/F_i,
\end{equation}
with the same ground set $E$.
\end{definition}

Since we assume that $M$ is loopless, the covectors of $M$ that are maximal with respect to the partial order are contained in $\{+, -\}^E$ and are called the \emph{topes} of $M$. The set of topes of $M$ is denoted by \[\mathcal{T}(M) \subseteq \{+, -\}^E.\] 

Let $S\subset E$, the restriction of a tope $T$ of $M$ to $\{+,-\}^S$ is always a tope of $M|_S$, and also a tope of $M/(E\setminus S)$ as long as $S$ is a flat. 
For a tope $T$ of $M$ and a subset $F\subseteq E$, the sign vector $T\setminus F$ is obtained from $T$ by setting all coordinates in $F$ to $0$. More precisely, 

$$ (T \setminus F)_e = 
\begin{cases} 
T_e \text{ \ \  if } e \not \in F \\
 0\text{ \ \ \  if } e \in F.
 \end{cases}
 $$ 
This is not to be confused with the operation of matroid deletion.

The composition operation $\circ$ on covectors $L$ and $K$ of $M$ is defined by 
$$(L \circ K)_e = 
\begin{cases} 
L_e \text{ \ \  if } L_e \neq 0\\
 K_e \text{ \ \ \  if } L_e = 0 .
 \end{cases}
$$
The tuple $L \circ K \in \{0, +, -\}^E$ is another covector of $M$.
Hence, $L \circ T$ is always a tope of $M$ for any covector $L$ and tope $T$.

\subsection{Salvetti Complexes}

Let $M$ be an oriented matroid on ground set $E$ of rank $d$.
Recall that the complement of the support $\supp(L)$ of a covector $L$ is a flat of the underlying matroid, and we define the dimension of a covector as the rank of such a flat.

The {\em Salvetti complexes} are topological models for the ``complexification'' of $M$: as stated in the introduction, when $M$ is an oriented matroid realised by a real hyperplane arrangement, the Salvetti complexes are deformation retracts of the complement of the complexification of the arrangement \cite{DDP_Sal, Sal87}; see also the alternative proof for simplicial arrangement case due to Paris \cite{Paris93}.

We describe two realisations of the Salvetti complex.
The finer version describes the complex as a simplicial complex.

\begin{definition}
The {\em fine Salvetti complex} $\widetilde{\Sal}_M$ of $M$ is the order complex of the {\em Salvetti poset} $$\{w(L,T):L\in\Lcal, T\in\Tcal, L\leq T\},$$
where $w(L,T)\leq w(L', T')$ if $L'\leq L$ and $T=L\circ T'$.
\end{definition}

In this paper, we most often use a coarser description of the space as a regular CW complex, which we denote as the (coarse) Salvetti complex.

\begin{proposition}\label{prop:salComplex} \cite[Proposition 5.50]{OT_book}
For every $L\leq T$, denote by $Z(L,T)$ the restriction of the fine Salvetti complex to the vertex set $\{w(L',L'\circ T):L'\geq L\}$.
Then $Z(L,T)$ is homeomorphic to a ball of dimension $\dim L$, whose boundary is equal to $\bigcup_{L'>L} Z(L',L'\circ T)$.
Furthermore, the collection of cells $\{Z(L,T):L\leq T\}$ forms a regular CW complex $\Sal_M$ homeomorphic to the fine Salvetti complex.
\end{proposition}

We elaborate more details and notations from the above proposition.
For any pair of $L\in\Lcal, T\in\Tcal$, define $Z(L,T)$ to be the cell $Z(L,L\circ T)$; it is consistent with the already defined case $L<T$.
Therefore, we have $Z(L,T) = Z(L,T')$ for distinct $T,T'$ if and only if $L \circ T = L \circ T'$.

The set of vertices of the Salvetti complex is $\{Z(T,T):T\in\Tcal\}$, hence there is a bijection between the set of vertices and $\mathcal{T}$. This is because $Z(T,T')$ is independent of the choice of $T'$ since $ Z(T,T') = Z(T,T \circ T') = Z(T,T)$.

\begin{definition}\label{def:conj}
The Salvetti poset, hence the fine Salvetti complex, admits a canonical involution which we call \emph{the conjugation action}, given by $\overline{w(L,T)}=w(L,L\circ(-T))$. The conjugation action is well-defined for the Salvetti complex, in which the conjugate of $Z(L,T)$ is simply $Z(L,-T)$.

The fixed locus of the conjugation action consists of cells $\{w(T, T) \ | \ T \in \mathcal{T} \}$ and $\{Z(T,T) \ | \ T \in \mathcal{T} \}$ in the fine Salvetti complex and Salvetti complex, respectively.
\end{definition}

Throughout we denote by $\Delta(P)$ the order complex of a poset $P$ whose vertex set is $P$ itself. 
Let $L$ be a covector and $F$ be the flat complement to $\supp(L)$.
The cell $Z(L,T)$ is canonically homeomorphic to the order complex $\Delta(\Lcal(M|_F))$ (i.e., a topological representation of the oriented matroid $M|_F$) in the following sense:
consider the barycentric subdivision of $Z(L, T)$ as a subcomplex of $\widetilde{\Sal}_M$, whose vertices are $w(L',L'\circ T)$'s for $L'\geq L$, then the homeomorphism is defined by the map of vertices $w(L',L'\circ T)\mapsto L'|_F$.
In particular, when the oriented matroid is realisable by a real hyperplane arrangement $\{{\bf v}_i^\perp: {\bf v}_i\in\mathbb{R}^d\setminus\{{\bf 0}\}\}$, each top dimensional cell $Z({\bf 0},T)$ has the same face poset as the {\em zonotope} of the arrangement, i.e., the Minkowski sum of the vectors ${\bf v}_i$, hence the notation $Z(L,T)$.

\begin{example} \label{ex:SalvettiU23}
\normalfont
Let $M$ be the oriented matroid realised by the real hyperplane arrangement in Example~\ref{ex:CovectorsU23} (and depicted in 
Figure \ref{diag:U23}).
The topes $\Tcal=\{T_1,\ldots,T_6\}$ are covectors of dimension $0$, the set 
$\Lcal^*=\{\zeta_1,\ldots,\zeta_6\}$ are the covectors of dimension 
$1$, and $O$ is of dimension $2$. Notice that the dimension of a covector $L$ corresponds to the dimension of the zonotope $Z(L,T)$.
The Salvetti complex of $M$ thus has six 2-dimensional cells $Z(O,T_i), T_i\in\Tcal$, each cell has the cellular structure of a hexagon (the Minkowski sum of three generic vectors in $\mathbb{R}^2$).

\begin{figure}[]
\begin{center}
    \includegraphics[scale=0.4]{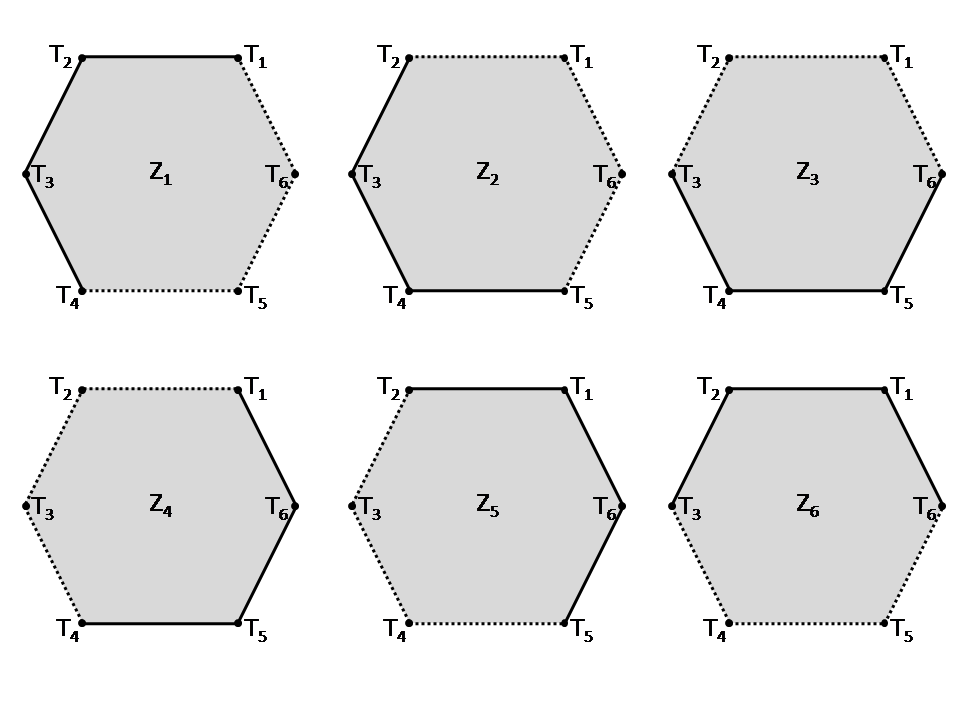}\\
  \vspace{-0.2cm}
  \caption{The six 2-dimensional cells of the Salvetti complex from Example \ref{ex:SalvettiU23}}
  \label{diag:Salvetti}
\end{center}
\end{figure}

The complex $\Sal_M$ has 6 vertices $Z(T_i,\bullet), T_i\in\Tcal$, which we just label by $T_i$ in the diagram.
It also has 12 1-cells $Z(\zeta_i,T)$, two for each $\zeta_i\in\Lcal^*$: recall that $T_i,\zeta_i,T_{i+1}$ are listed in counter-clockwise order in Figure~\ref{diag:U23}, and $\zeta_i$ gives rise to $Z(\zeta_i,T_i),Z(\zeta_i,T_{i+1})$, whose boundaries both consist of the vertices $Z(T_i,\bullet),Z(T_{i+1},\bullet)$. In the diagram,  $Z(\zeta_i,T_i)$ is drawn as a solid segment whereas $Z(\zeta_i,T_{i+1})$ is dotted.

The complex $\Sal_M$ is obtained by identifying the pieces in Figure \ref{diag:Salvetti} by the corresponding labels.
\end{example}

\subsection{The Orlik--Solomon Algebra and the Cordovil Algebra} \label{sec:OS_Cord}

Next we introduce the Orlik--Solomon algebra of a matroid.
Its significance is that when the matroid is representable over $\CC$, its Orlik--Solomon algebra is naturally isomorphic to the (de Rham) cohomology ring of the complement of any complex hyperplane arrangement representing the matroid.

\begin{definition}
For a matroid $\underline{M}$, its {\em Orlik--Solomon algebra} $\OS^\bullet(\underline{M}; R)$ (where $R=\ZZ$ or $\Ztwo$) is the algebra defined by \[ \frac{\bigwedge(R^*)^E}{\langle \sum_{k=1}^t (-1)^k e^*_{i_1}\wedge\ldots\wedge\widehat{e^*_{i_k}}\wedge\ldots\wedge e^*_{i_t}\rangle},\] where the generators of the ideal are taken over all circuits $\{e_{i_1}<\ldots<e_{i_t}\}$ of $\underline{M}$, ordered by an arbitrary but fixed ordering of $E$.
\end{definition}

The Orlik--Solomon algebra is naturally graded and $\OS^\bullet(\underline{M}; \ZZ)$ is torsion free. 
Denote the rank of the degree $p$ piece of $\OS^{p}(M; \ZZ)$ by $b_p(\underline{M})$, or simply $b_p$ if the underlying matroid is clear. Then $b_p(\underline{M}) = \dim\OS^p(\underline{M}; \Ztwo)$. 

\begin{theorem} \cite[Section~7]{BZ92} \label{thm:H_eq_OS}
Let $M$ be an oriented matroid.
Then 
$$
H^\bullet(\Sal_M;\ZZ)\cong \OS^\bullet(\underline{M}; \ZZ).
$$
In particular, $b_p(\underline{M})=\dim H^p(\Sal_M;\Ztwo)=\rank H^p(\Sal_M;\ZZ)$.
\end{theorem}

By a theorem of Zaslavsky \cite{GZ83}, usually stated using the {\em characteristic polynomial} of $\underline{M}$, the number of topes of an oriented matroid $M$ is equal to the sum of $b_p$'s.
\begin{equation}\label{Zaslavsky}
|\Tcal(M)| = \dim \OS^\bullet(\underline{M}; \Ztwo)=\rank \OS^\bullet(\underline{M}; \ZZ).
\end{equation}

When a statement holds for both $R = \ZZ$ or $\Ztwo$ we drop $R$ from the notation $\OS_p(\underline{M}; R)$.
We also require the duals of the graded pieces of the Orlik--Solomon algebra $\OS_p(\underline{M}) := \OS^p(\underline{M})^*$. From the definition of $\OS^\bullet(\underline{M})$ it follows that $\OS_p(\underline{M})$ is a vector subspace/submodule of $\bigwedge^p R^E$. 

For $i \in E$ let $\ee_i$ denote the standard basis element of $R^E$ corresponding to $i$.
For a subset $F \subset E$ denote by $\ee_F = \sum_{i \in F} \ee_i$. Then we have the following description of generators of $\OS_p(\underline{M})$. The statement below is equivalent to the one in \cite{Zhar_OS}, which is expressed in terms of faces of the polyhedral fan of the matroid.

\begin{proposition}\cite[Lemma~5]{Zhar_OS}\label{prop:OSZhar}
For a matroid $M$ 
\[\OS_p(\underline{M}) = \langle  \ee_{F_1} \wedge \dots \wedge \ee_{F_p} \ | \ F_1 \subsetneq \dots \subsetneq F_p \rangle, \]
where $F_1 \subsetneq \dots \subsetneq F_p$ is a flag of flats in $\underline{M}$ where $\rank F_i=i$.
\end{proposition}

In Section~\ref{sec:cosheaf}, we will work with another algebra defined for an oriented matroid, first introduced by Cordovil \cite{Cord02} as a combinatorial analogue of the {\em Orlik--Terao algebra} of a real hyperplane arrangement.

\begin{definition} \label{def:Cordovil}
For an oriented matroid $M$, its {\em Cordovil Algebra} $\Acal^\bullet(M; R)$ is the naturally graded algebra
\[ \frac{R[x^*_i:i\in E]}{\langle (x^*_i)^2\rangle+\langle \sum_{k=1}^t C(i_k) x^*_{i_1}\ldots\widehat{x^*_{i_k}}\ldots x^*_{i_t}\rangle},\] where the generators of the second ideal are taken over all circuits $\{e_{i_1},\ldots,e_{i_t}\}$ of $\underline{M}$, and $C:E\rightarrow\{+,-,0\}$ is a signed circuit of $M$ supported on it.
\end{definition}

It is evident that $\OS^\bullet(\underline{M}; \Ztwo)\cong\Acal^\bullet(M;\Ztwo)$ as algebras via the isomorphism $e^*_{i_1}\wedge\ldots\wedge e^*_{i_k}\mapsto x^*_{i_1}\ldots x^*_{i_k}$. Moreover, it follows easily from the deletion-contraction recurrence \cite[Theorem~2.7]{Cord02} for the Cordovil algebra that it is torsion free, thus $\rank \Acal^p(M;\ZZ)=b_p(\underline{M})$ as well. However, $\OS^\bullet(\underline{M}; \ZZ)$ is anti-commutative while $\Acal^\bullet(M;\ZZ)$ is commutative; more importantly, $\Acal^\bullet(M;\ZZ)$ records the oriented matroid data but $\OS^\bullet(\underline{M}; \ZZ)$ does not.

The Cordovil algebra is less studied than the Orlik--Solomon algebra in the literature, and some basic (but novel to the best of our knowledge) properties of the Cordovil algebra and their duals are proven in Section~\ref{sec:Cordovil}. Nevertheless, we note a topological interpretation of $\Acal^\bullet(M)$ is given in \cite{Mose17} and extended in \cite{DPW_COM}; the connection of our work with these interpretations is to be explored. 

\subsection{Matroid Fans and Cosheaves}

In Section \ref{sec:cosheaf}, we will show that the Varchenko--Gelfand filtration of the $\ZZ$-tope space of an oriented matroid as well its Cordovil algebra can be made into cosheaves on the fan of the underlying matroid. 

We first present the fan of a matroid as introduced by Ardila and Klivans \cite{AK06}. 

\begin{definition}\label{def:matfan}
The fan of a matroid $\Sigma_{\underline{M}}$ is the polyhedral fan in $\RR^E$ consisting of the collection of cones 
$$\{ \sigma_{\mathcal{F}} \ | \ \mathcal{F} = \emptyset \subsetneq F_1 \subsetneq \dots \subsetneq F_k \subsetneq E \},$$ 
where $\mathcal{F}$ is a flag of flats of $\underline{M}$ and 
$\sigma_{\mathcal{F}} = \langle \ee_{F_1} \dots, \ee_{F_k} \dots, \pm \ee_E\rangle_{\geq 0}$.
\end{definition}

The projective matroid fan $\Sigma^{\pj}_{\underline{M}}$ is the image of the fan above under the quotient map 
$\RR^E \to \RR^E / \langle (1 ,\dots, 1) \rangle.$ Note that there is a bijection between the cones of $\Sigma_M$ and $\Sigma^{proj}_{\underline{M}}$. 

To describe sheaves and cosheaves on polyhedral fans, we will consider them as (finite) categories. The objects of the category of $\Sigma$ are its cones and the morphisms correspond to inclusions of cones. In other words, there is a map between the cones $\tau \to \sigma$ if and only if there is an inclusion $\tau  \subset \sigma$. The opposite category obtained by reversing the arrows is denoted by $\Sigma^{\text{op}}$. 

\begin{definition}\label{def:sheaf}
A sheaf on a polyhedral fan $\Sigma$ is a functor 
$\mathfrak{F} : \Sigma \to {\rm Mod_R}$
and a cosheaf on a polyhedral fan $\Sigma$ is a functor 
$\mathfrak{G} : \Sigma^{\text{op}} \to {\rm Mod}_R$.
\end{definition}

In the next section, we define the tope space of an oriented matroid. We define the $\ZZ$-sign cosheaf $\Scal$ of an oriented matroid as the functor which assigns to each face of the fan of the underlying matroid the $\ZZ$-tope space of $M_\Fcal$, where $\Fcal$ is the flag of flats corresponding to the face of the fan. The morphisms of the $\ZZ$-sign cosheaf are inclusions. 

\section{Three Filtrations of the Tope Space}

This section presents three different filtrations of the tope space of an oriented matroid. 

\begin{definition}
The \emph{$R$-tope space} of an oriented matroid $M$ is free module $R[\Tcal(M)]$.
\end{definition}

Notice that the $\Ztwo$-tope space is a vector subspace of the $2^{|E|}$-dimensional group algebra 
$(\Ztwo)[(\Ztwo)^E]$, and the $\ZZ$-tope space is a submodule of $\ZZ [(\Ztwo)^E]$.

\subsection{Dual Varchenko--Gelfand Filtration} \label{sec:GV}

Let $M$ be an oriented matroid with topes $\Tcal$. The dual Varchenko--Gelfand filtration filters the $\ZZ$-tope space. In fact, it is most easily defined on the dual of the tope space (in which the filtration is commonly referred as Varchenko--Gelfand filtration). 
The collection of functions $\ZZ[\Tcal]^*:=\{f:\ZZ[\Tcal]\rightarrow\ZZ: f \text{ linear}\}$ is equipped with a ring structure given by pointwise addition and multiplication, known as the {\em Varchenko--Gelfand ring}.

\begin{definition}
For each $e\in E$, define the {\em Heaviside function} $h_e\in\ZZ[\Tcal(M)]^*$ by extending $h_e(T)=
\begin{cases}
  1  & T_e= + \\
  0 & T_e = -
\end{cases}$
linearly.
\end{definition}

The ring $\ZZ[\Tcal]^*$ is then generated by $h_e$, $e\in E$, and the unit ${\bf 1}$ given by ${\bf 1}(T) = 1$ for all $T \in \Tcal$.
Now for each non-zero $f\in \ZZ[\Tcal]^*$, we can define its {\em degree} as the minimum degree of polynomial in $h_e$'s that represents $f$. We define the degree of the zero function as $-1$.
The degree of any element of $\ZZ[\Tcal(M)]^*$ is well-defined and always at most $d$ \cite{GV87}.

\begin{definition}
Denote by $\Pcal^p$ the collection of functions in $\ZZ[\Tcal(M)]^*$ whose degree is at most $p$.
The {\em degree filtration} of $\ZZ[\Tcal(M)]^*$ is the filtration $$\ZZ[\Tcal(M)]^*=\Pcal^d(M)\supset\Pcal^{d-1}(M)\supset\ldots\supset\Pcal^{-1}(M)=\{0\}.$$
The {\em dual degree filtration} of $\ZZ[\Tcal(M)]$ is $$\ZZ[\Tcal(M)]=\Pcal_0(M)\supset\Pcal_1(M)\supset\ldots\supset\Pcal_d(M)\supset\Pcal_{d+1}(M)=\{0\},$$ where $\Pcal_p(M)$ is the annihilator of $\Pcal^{p-1}(M)$ with respect to the pairing $$\ZZ[\Tcal(M)]^*\times\ZZ[\Tcal(M)]\rightarrow\ZZ$$ given by $(f,\gamma)\mapsto f(\gamma)$.

Denote by $\overline{\Pcal}_\bullet(M)$ the filtration of $(\Ztwo)[\Tcal(M)]$ given by the mod 2 reduction of $\Pcal_\bullet$.
\end{definition}

Next we give an alternative description of the dual degree filtration of Varchenko--Gelfand using {\em prefix chains}.
The notion was first introduced in the context of hyperplane arrangements by Varchenko--Gelfand under the name of {\em flag cochains}, and reintroduced by Denham for general oriented matroids under the name of {\em bricks}.

For an element ${\bf v} \in (\Ztwo)^E$, let $T_{{\bf v}} \in \{+,-\}^E$ be given by $(T_{\bf v})_i = (-1)^{{\bf v}_i}$, i.e., $(T_{\bf v})_i=+$ if and only if ${\bf v}_i=0$.

\begin{definition}\label{def:tope_flag}
Let $\Fcal:\emptyset=F_0 \subsetneq  F_1 \subsetneq \ldots \subsetneq  F_k=E$ be a flag of flats.
Define 
$$\Tcal_\Fcal := \{ T \in \mathcal{T}(M) \ | \ T\setminus F_i \in \mathcal{L} \text{ for all } F_i \in \mathcal{F}\}. $$ 
\end{definition}

From the point of view of a pseudohyperplane arrangement, $\Tcal_\Fcal$ consists of the topes (viewed as connected components of the complement of the arrangement) for which the intersection with the (pseudo)subspace corresponding to each flat $F_i$ is a cell of (the closure of) the tope of codimension $\rank(F_i)$.
The next proposition gives a more explicit description of $\Tcal_\Fcal$.

{\sc\bf Notation}: For the rest of this paper, when a complete flag of flats $\Fcal$ is given, we denote by $\dd_{\Fcal,p}$ (respectively, $\ff_{\Fcal, p}$) the vector $\ee_{F_p\setminus F_{p-1}}$ (respectively, $\ee_{F_p}$); when the flag is clear we often omit the $\Fcal$ subscript.

\begin{proposition} \label{prop:T_F_structure} \cite[Lemma 3.2]{RRS}
Let $\Fcal:\emptyset=F_0\subsetneq  F_1 \subsetneq \ldots \subsetneq F_k=E$ be a flag of flats of $M$. Then $\Tcal_\Fcal$ is the collection of topes of $M_\Fcal$.
If the flag is complete, then $\Tcal_\Fcal$ has the structure of an affine subspace of $(\Ztwo)^E$ which is parallel to $\langle \ff_1, \dots, \ff_d\rangle=\langle \dd_1, \dots, \dd_d\rangle$.
Namely, $\Tcal_\Fcal$ is in correspondence with $\bf v + \langle \ff_1, \dots, \ff_d\rangle,$
for any $T_{\bf v} \in \Tcal_\Fcal$. 
\end{proposition}

See Example~\ref{ex:prefix_chains} for an illustration.
Suppose $V$ is a $p$-dimensional vector subspace of $ (\Ztwo)^E$ with a distinguished basis $\Bcal = \{ {\bf v}_1, \dots, {\bf v}_p\}$. 
Then for all ${\bf v} \in V$ we can write 
${\bf v} = \sum_{i = 1}^p a_i{\bf v}_i$ where $a_i = 0$ or $ 1$. 
Define the weight of ${\bf v }$ with respect to $\Bcal$ to be the number of $a_i$'s equal to one. Denote this weight by $w_\Bcal({\bf v})$.
Let $\text{Aff}_p(V)$ denote the set of $p$-dimensional affine subspaces of a vector space $V$.

If $U$ is an affine subspace of $ (\Ztwo)^E$, and $\vv\in U$, again upon choosing a distinguished basis $\Bcal$
of the vector subspace $V =\vv + U $ as above, we define $w_{\Bcal, \vv} (\uu) = w_\Bcal(\uu+ \vv)$. 
For a different choice of $\vv' \in U$, we have \[(-1)^{w_{\Bcal, \vv}(\uu)} = \epsilon (-1)^{w_{\Bcal, \vv'}(\uu)},\] for all $ \uu\in U$, where $\epsilon=(-1)^{w_\Bcal(\vv+\vv')}$ is independent of $\uu$.

\begin{definition} \label{def:prefix_chain}
Let $U$ be an affine subspace of $(\ZZ/2\ZZ)^E$ with a fixed choice of basis $\Bcal$ and $\vv\in U$.
Then define 
$$\gamma_{U, \Bcal, \vv} = \sum_{\uu\in U} (-1)^{w_{\Bcal, \vv}(\uu)} T_\uu\in\ZZ[(\ZZ/2\ZZ)^E].$$
Here, for an element $\uu$ of $(\ZZ/2\ZZ)^E$, we denote by $T_\uu$ the variable corresponding to $\uu$ in the group ring, which is consistent with the notation for tope spaces.
From the previous discussion, changing the choice of $\vv$ only changes $\gamma$ up to a sign.
\end{definition}

\begin{definition}\label{def:prefixchains}
Fix a complete flag of flats $\Fcal:\emptyset=F_0 \subsetneq  F_1 \subsetneq \ldots \subsetneq  F_d=E$, and fix a tope $T_{\vv}\in\Tcal_\Fcal$.
Let $U_{\Fcal,\vv,p} \in \text{Aff}_p (\ZZ/2\ZZ)^E$ be given by $\vv + \langle \dd_1, \ldots, \dd_p\rangle$. 

Then the {\em $p$-th prefix chain} of $\Fcal$ with respect to $T_{\vv}$ is $$\gamma_{\Fcal ,\vv, p} :=
\gamma_{U_{\Fcal,\vv,p},\Bcal, \vv}\in\ZZ[\Tcal],$$
where $\Bcal = \{\dd_1, \ldots, \dd_p\}$.
Occasionally we write $\gamma_{\Fcal,T,p}$ for the chain $\gamma_{\Fcal,\vv,p}$ with $T=T_{\vv}$.
\end{definition}

Notice that the topes involved in $\gamma_{\Fcal,\vv,p}$ (respectively, vectors in $U_{\Fcal,\vv,p}$) are precisely those topes of $\Tcal_\Fcal$ that agree with $T_{\vv}$ outside of $F_p$, and changing the origin $T_{\vv}$ to any other tope in the sum only changes the prefix chain up to sign.
Moreover, when $p<d$, a prefix chain only depends on the partial flag up to $F_p$ and the choice of $T_{\vv}$.

\begin{example} \label{ex:prefix_chains} \rm
Following Example~\ref{ex:SalvettiU23}, consider the flag $\Fcal: \emptyset \subsetneq  \{\ell_1\} \subsetneq  E$. The set 
$\Tcal_\Fcal$ consists of the topes $T_1,T_2,T_4,$ and $T_5$.
Choosing $T_1$ as the origin, the prefix chains $\gamma_{\Fcal,T_1,1}$ and $\gamma_{\Fcal,T_1,2}$ equal $T_1-T_2$ and $T_1-T_2+T_4-T_5$, respectively.
If we choose $T_5$ as the origin instead, then $\gamma_{\Fcal,T_5,1}=T_5-T_4,\gamma_{\Fcal,T_5,2}=T_5-T_4+T_2-T_1$.
Note that $\gamma_{\Fcal,T_1,2}$ and $\gamma_{\Fcal,T_5,2}$ differ only by a sign.
\end{example}

\begin{proposition}\cite{GV87} \label{prop:prefixdualdegree}
For an oriented matroid $M$, the $p$-th part of the dual degree filtration is the subgroup spanned by its $p$-th prefix chains. 
In other words,
$$\Pcal_p(M) = \langle \gamma_{\Fcal ,\vv, p}  \ | \  \Fcal \text{ complete flag of flats} \rangle .$$
\end{proposition}

As a simple but helpful illustration, we have \[\Pcal_1(M)=\{\sum_{T\in\Tcal} c_T\cdot T:\sum_{T\in\Tcal} c_T=0\}.\]
We have the $\subset$ containment because each $\gamma_{\Fcal,\vv,p}$ satisfies the linear condition on the right hand side. For the reversed containment, recall that two topes $T,T'$ are {\em adjacent} if they differ over a rank 1 flat $F_1$, hence $T-T'=\gamma_{\Fcal,T,1}$ for any flag $\Fcal$ extending $F_1$.
Since any two topes $T,T'$ are connected by a sequence of adjacency relations \cite[Lemma~4.4.1]{BLSWZ}, the difference $T-T'$ is a sum of $1$-st prefix chains and is in $\Pcal_1$. It is clear that the differences $T-T'$ span the whole right hand side. 

\begin{remark} \label{rem:asym}
\rm
As mentioned in the introduction, another filtration of interest in this picture is the {\em asymptotic filtration}, which we now briefly describe.

The filtration is defined using the {\em Schechtman--Varchenko bilinear form}, which is the bilinear form $B_M(\cdot,\cdot)$ defined over $\ZZ[x_e:e\in E][\Tcal(M)]$ by extending $B_M(T,T')=\prod_{e: T(e)\neq T'(e)}(1+x_e)$ bilinearly (we use a change of variables differs from the usual one for brevity).
The $p$-th piece of the asymptotic filtration of $\ZZ [\Tcal(M)]$ consists of all chains $\gamma$ with the property that, for every $T'\in\Tcal(M)$, the polynomial $B(\gamma, T')\in\ZZ[x_e: e\in E]$ has no terms of degree less than $p$.
\end{remark}

\subsection{Kalinin Filtration} \label{sec:Kalinin}

The $\Ztwo$-tope space can also be filtered via the Kalinin filtration. 
For a topological space $X$ equipped with an involution $c : X \to X$, Kalinin defined a spectral sequence which starts from the homology of $X$ and converges to the homology of the fixed locus $\text{Fix}(c) \subset X$ \cite{Kalinin}, see also \cite{Degtyarev}. 

The first page of Kalinin spectral sequence has terms $E^1_{\bullet} = H_{\bullet}(X)$ and differentials $d^1_{\bullet} = 1 + c_*$.
The further pages have terms $$E^{r}_{p} = \ker d^{r-1}_p / \Image d^{r-1}_{p-r+1}.$$
A cycle $x_p$ is in $\ker d^r_p \subset H_p(X)$ if and only if there exist chains $$y_p =x_p, y_{p+1} , \dots, y_{p+r-1} $$ so that $\partial y_{i+1} = (1+c_*) y_i. $ When such chains exist, the differential on the $r$-th page is the map $d^r_{p}$ defined on the representative $x_p$ by $d_p^r x_p = (1 + c_*) y_{p+r-1}$.

In our context, the space $X$ under consideration is the Salvetti complex appearing in Proposition \ref{prop:salComplex} and the involution is conjugation action from Definition \ref{def:conj}. Thus the fixed locus consists of the collection of vertices of the Salvetti complex in correspondence with the topes. Hence the fixed locus only has homology in degree $0$ and $H_{\bullet}(\text{Fix}(c); \Ztwo) = H_{0}(\text{Fix}(c); \Ztwo) = \Ztwo[\mathcal{T}(M)]$. We use this to simplify the presentation of the Kalinin spectral sequence below and obtain the Kalinin filtration. 

\begin{proposition}\label{prop:Kalinin}
For the Salvetti complex of an oriented matroid, the Kalinin spectral sequence degenerates at the first page. Moreover, associated gradeds of the spectral sequence induce a filtration of the $\Ztwo$-tope space: 
$$\Kcal_d(M) \subset \dots \subset \Kcal_1(M) \subset \Kcal_0(M) = \Ztwo[\Tcal(M)]$$ such that $\Kcal_p(M) / \Kcal_{p+1}(M)$ is isomorphic to $H_p(\Sal_M; \Ztwo)$ for all $p$.
\end{proposition}

\begin{proof}
The dimension of $H_{\bullet}(\text{Fix}(c); \Ztwo)$ is equal to the number of topes of $M$, which by Zavlasky's theorem (Equation (\ref{Zaslavsky})) is equal to the total dimension $H_{\bullet}(\Sal_M; \Ztwo)$.
Hence $\Sal_M$ is maximal in the sense of the Smith--Thom inequality, and 
\[ |\Tcal(M)| = \dim H_{\bullet}(\text{Fix}(c); \Ztwo) = \dim  H_{\bullet}(\Sal_M; \Ztwo). \]
Therefore, the maps $d_{\bullet}^p$ must all be zero and the Kalinin spectral sequence degenerates at the first
page. 
Since $\Ztwo[\Tcal(M)]$ is canonically isomorphic to $H_{\bullet}(\text{Fix}(c); \Ztwo)$ we obtain a filtration of the tope space. 
\end{proof}

The isomorphism between the intermediate gradeds of the filtration and the homology groups of the Salvetti complex from the above proposition are given by the Viro homomorphisms \cite{Degtyarev} as below.

\begin{definition} \label{def:Kalinin_Z2Z}
Let $M$ be an oriented matroid of rank $d$ and let $\mathcal{T}$ denote its collection of topes. 
The {\em (Borel--)Viro homomorphisms}
$$bv_p:\Kcal_{p}(M) \rightarrow H_p(\Sal_M;\Ztwo)$$ are defined recursively by 

\begin{itemize}
\item $bv_0$ is the induced inclusion map $i_{\ast} : H_0(\text{Fix}(c) ;\Ztwo) \to H_0(\Sal_M; \Ztwo)$,
\item $bv_p:\Kcal_{p}(M) \rightarrow H_p(\Sal_M;\Ztwo)$ 
is given by $bv_p(\gamma) = 
[\beta_p+\overline{\beta_p}]$,
where $\beta_p$ is any $p$-chain whose boundary is $bv_{p-1}(\gamma)$ in $\Sal_M$.
\end{itemize}
\end{definition}

We have $\gamma \in\Kcal_p(M)$ if for $i = 1, \ldots , p$ there exists $i$-chains $\beta_i \in C_i(\Sal_M; \Ztwo)$ such that $\partial\beta_1= i_{\ast}\gamma$ and $\partial\beta_i=\beta_{i-1}+\overline{\beta_{i-1}}$ for $i=2,\ldots,p$.

The Viro homomorphism descends to a map $$bv_p:\Kcal_p(M)/ \Kcal_{p+1} (M) \rightarrow H_p(\Sal_M;\Ztwo).$$ This arises as the isomorphism of the $E^{\infty}$ page of the Kalinin spectral sequence with the associated gradeds. This isomorphism has the above form since the Kalinin spectral sequence degenerates at the first page by Proposition \ref{prop:Kalinin}, and the homology of $\text{Fix}(c)$ is only non-zero in dimension $0$. 
To see directly that the definition of $bv_p$ does not depend on the choices of $\beta_i$'s:
suppose by induction that $\beta_{i-1}+\overline{\beta_{i-1}}$ is well-defined, then any two choices of $\beta_i$ differ by some homology class, while by \cite[Corollary~A.2]{Wil78}, the conjugation acts as an identity on $H_i(\Sal_M;\Ztwo)$, so the difference of the two candidates $\beta_i$ cancels out upon adding to their conjugates.

\begin{example} \rm
In the Salvetti complex corresponding to the oriented matroid of the arrangement in Example \ref{ex:SalvettiU23}, the topes $T_1$ and $T_2$ differ by exactly by the first coordinate, and correspond to two vertices in $\text{Fix}(c) \subset \Sal_M$.
The covector $\zeta_1$ has $0$ in its first coordinate and otherwise agrees with $T_1$ and $T_2$. 
The $1$-dimensional cells $Z(\zeta_1,T_1)$ and $Z(\zeta_1,T_2)$ satisfy $\partial Z(\zeta_1,T_1) = \partial Z(\zeta_1,T_2) = T_1 + T_2 \in C_0(\Sal_M; \Ztwo)$ and also $\overline{Z(\zeta_1,T_1)} = Z(\zeta_1,T_2)$.
Therefore, to find $bv_1(\gamma)$ we can choose $\beta_1$ to be either $Z(\zeta_1,T_1)$ or $Z(\zeta_1,T_2)$ and $bv_1(\gamma) = Z(\zeta_1,T_1) + Z(\zeta_1,T_2) \in H_1(\Sal_M; \Ztwo)$.

Analogously, $bv_1(T_1+T_2+T_4+T_5)=[Z(\zeta_1,T_1)+Z(\zeta_1,T_2)+Z(\zeta_4,T_4)+Z(\zeta_4,T_5)]$. In order to find $bv_2(T_1+T_2+T_4+T_5)$, we can pick $\beta_2=Z(O,T_1)+Z(O,T_2)$, and $bv_2(T_1+T_2+T_4+T_5)=[Z(O,T_1)+Z(O,T_2)+Z(O,T_4)+Z(O,T_5)]$.

We can also consider Kalinin spectral sequence in the situation of a real hyperplane arrangement and its complexification. 
Let $\mathbb{C} \Ccal$ denote the complement of the complexification of the arrangement from Example \ref{ex:SalvettiU23}. The real complement $\mathbb{R} \Ccal$ is the fixed locus of the action of complex conjugation acting on $\mathbb{C} \Ccal$, and consists of the six regions depicted in Figure \ref{diag:U23}. 
Consider points $x = (d, c) \in A$ and $ y = (-d, c) \in B$, where here $T_1$ and $T_2$ are denoted the regions labeled in Figure \ref{diag:U23}. Let $\gamma = x+y \in C_0(\mathbb{R}^2; \Ztwo)$. Note that $\gamma$ is non-zero in $H_0(\mathbb{R} \Ccal; \Ztwo)$.
Then $ \beta_1 = \{ ( e^{i\theta}, c) \ | \ \theta \in [0, \pi] \}$ 
forms a $1$-chain in the complexification 
whose boundary is $\gamma$. Therefore, $\gamma$ is zero in $H_0(\mathbb{C} \Ccal; \Ztwo)$ and so $bv_0 (\gamma) = 0$.
Taking the sum of $\beta_1$ with its complex conjugate yields
$\beta_1 + \overline{\beta_1} = \{ (e^{i\theta}, c) \ | \ \theta \in [0, 2\pi ) \}$, which is now a closed $1$-chain, as it is a circle embedded in $\mathbb{C} ^2$. Therefore, $bv_1(\gamma) = \beta_1 + \overline{\beta_1} \in H_1( \mathbb{C} \Ccal; \Ztwo)$. 
\end{example}

\subsection{Quillen Filtration} \label{sec:Quillen}

Recall that the set $\{+, -\}^E$ carries a vector space structure via the bijection $( \Ztwo)^E \to \{+, -\}^E$, given by ${\bf v} \mapsto   T_{\bf v}$ where $(T_{\bf v})_e := (-1)^{{\bf v}_e}$. Here we abuse notation and identify $\pm 1$ with $\pm$. 
Notice that the additive structure on $\Ztwo$ translates to the multiplicative structure on $\{+, -\}$.
\begin{definition}
Let $V$ be a $d$-dimensional vector space over $\FF_2$.
Consider the group algebra $\Ical_0(V):=(\ZZ/2\ZZ)[V]=(\ZZ/2\ZZ)[T_\vv:\vv\in V]$, where multiplication is given by $T_{\vv_1}T_{\vv_2}=T_{\vv_1+\vv_2}$.
For a subset $G\subset V$, denote by $\gamma_G$ the sum $\sum_{\vv\in G} T_\vv$.

Define the {\em Quillen filtration} of $\Ical_0(V)$ as follows: $$\Ical_1(V) = \{\sum_\vv a_\vv T_\vv: \sum_\vv a_\vv=0\}$$ is the {\em augmentation ideal} of $V$, whereas further pieces are defined as $$\Ical_p(V):=\Ical_1^p(V)=\langle t_1 t_2\ldots t_p: t_i\in\Ical_1(V)\rangle.$$
\end{definition}

\begin{proposition} \cite[Lemma~4.1 and Proposition~4.3]{RS} \label{prop:SS_generate_Quillen}
The vector space $\Ical_p(V)$ is additively generated by $\gamma_G$'s, ranging over all $p$-dimensional subspaces $G$ of $V$.
Furthermore, $\dim\Ical_p(V) - \dim\Ical_{p+1}(V) = \dim\bigwedge^p V = \binom{d}{p}$.
\end{proposition}

By Proposition~\ref{prop:T_F_structure}, given a complete flag $\Fcal$, after selecting an origin, we can identify $\Tcal_\Fcal$ with a vector subspace and consider the Quillen filtration $\Ical_p(\Tcal_\Fcal)$'s.
By \cite[Lemma~4.4]{RS}, the filtration obtained does not depend on the choice of origin. 

\begin{definition}
Let $M$ be an oriented matroid with topes $\Tcal$.
Denote $\Qcal_0(M):=(\ZZ/2\ZZ)[\Tcal]$. 

Using the natural inclusion of $(\ZZ/2\ZZ)[\Tcal_\Fcal]$ into $\Qcal_0(M)$, we can define the Quillen filtration on $\Qcal_0(M)$ by $$\Qcal_p(M):=\sum_\Fcal \Ical_p(\Tcal_\Fcal),$$
where $\Fcal$ ranges over all complete flags of flats of $M$.
\end{definition}

By Proposition \ref{prop:SS_generate_Quillen}, the Quillen filtration of the tope space is equivalently
$$\Qcal_p(M) = \langle \gamma_{ U} \ | \ U \in \text{Aff}_p(\Tcal_{\Fcal}) \rangle,$$
where $\Fcal$ ranges over all complete flags of flats of $M$, $\text{Aff}_p(\Tcal_{\Fcal})$ denotes the set of affine subspaces of the affine space $\Tcal_{\Fcal}$
 of dimension $p$, and $\gamma_U =\sum_{\mathbf{u} \in U} T_{\mathbf{u}} \in \Ztwo[\Tcal(M)]$. 
 Note that $\gamma_U$ is the mod $2$ reduction of $\gamma_{U, \mathcal{B}, \vv}$ for any choice of $\mathcal{B}$ and $\vv$.

 In addition, each piece of the Quillen filtration comes with a map to the dual of the Orlik--Solomon algebra as described in Proposition \ref{prop:OSZhar}. These maps are denoted by $\qbv_p \colon \Qcal_p(M) \to \OS_p(\underline{M}; \Ztwo)$ and are defined on the generators 
 $\{\gamma_{ U} \ | \ U \in \text{Aff}_p(\Tcal_{\Fcal}) \} $
 via 
 \begin{equation} \label{eq:defqp} 
 \qbv_p(\gamma_{ U}) = \vv_1 \wedge \dots \wedge \vv_p, 
 \end{equation}
 where $\vv_1, \dots, \vv_p \in (\ZZ/2\ZZ)^E$ form a basis of the tangent space of the affine space $U$.

\begin{theorem} \label{thm:sign_cosheaf_dim}
The maps ${\qbv}_p \colon \Qcal_p(M) \to \OS_p(\underline{M}; \Ztwo)$
are well-defined and descend into isomorphisms between $\Qcal_p(M)/\Qcal_{p+1}(M)$ and $\OS_p(\underline{M};\Ztwo)$.
In particular, for each $p$, \[\dim\Qcal_p(M) - \dim\Qcal_{p+1}(M)= b_p(M).\]
\end{theorem}

\begin{proof}
The proof follows closely the arguments given in \cite[Proposition~4.3 and Lemma~4.8]{RS}, which were applied for the uniform matroid $U_{n, n+1}$ and were shown to work for all oriented matroids in \cite[Proposition 5.7]{RRS2}. 
Since the statement in the last reference is a stronger result in terms of cellular cosheaves, here we overview the arguments and leave the reader to \cite{RS} and \cite{RRS2} for the detailed calculation.

It can be shown that the map $\vv_1 \wedge \dots \wedge \vv_p\mapsto [\gamma_U]\in\Ical_p(\Tcal_\Fcal)/\Ical_{p+1}(\Tcal_\Fcal)$, where $U$ is the span of $\vv_i$'s, induces an isomorphism between between $\bigwedge^p\Tcal_\Fcal$ and $\Ical_p(\Tcal_\Fcal)/\Ical_{p+1}(\Tcal_\Fcal)$ (indeed, the isomorphism is a linear algebra fact that remains valid when $\Tcal_\Fcal$ is replaced by any vector space over $\Ztwo$).
Summing the isomorphisms over all complete flags $\Fcal$ and applying Proposition~\ref{prop:OSZhar} gives a map between $\OS_p(\underline{M};\Ztwo)$ and $\Qcal_p(M)/\Qcal_{p+1}(M)$; the map is surjective and thus an isomorphism by a dimension count.
The map $\qbv_p$ can be described as the composition of the projection map $\Qcal_p(M)\rightarrow\Qcal_p(M)/\Qcal_{p+1}(M)$ and the {\em inverse} of this isomorphism, hence it is well-defined, and the isomorphism claim is tautological.
\end{proof}

\begin{example} \label{ex:Quillen_no_Z} \rm
While the Quillen filtration can be defined for the group algebra with coefficients from any commutative ring, if taken over $\ZZ$ it does not provide a filtration with ranks corresponding to those over $\Ztwo$.

Consider the uniform matroid $M=U_{2,2}$, which has four topes $T_{00},T_{01},T_{10},T_{11}$.
The augmentation ideal in the Quillen filtration over $\ZZ$ for each complete flag $\Fcal$ is spanned by $T_{00}-T_{01},T_{00}-T_{10}, T_{00}-T_{11}$, hence the corresponding $\Qcal_2(\Tcal_\Fcal)$ contains
\begin{eqnarray*}
(T_{00}-T_{01})(T_{00}-T_{10})=T_{00}-T_{01}-T_{10}+T_{11},\\ 
(T_{00}-T_{01})(T_{00}-T_{11})=T_{00}-T_{01}+T_{10}-T_{11},\\
(T_{00}-T_{10})(T_{00}-T_{11})=T_{00}+T_{01}-T_{10}-T_{11},
\end{eqnarray*}
which span a subgroup of rank 3 instead of 1.
In fact, one can see that the filtration never stabilises.
\end{example}

\section{Comparison of the Filtrations}

Our goal in this section is to show that all three filtrations of the $\Ztwo$-tope space of an oriented matroid are the same.
Our strategy is to first show that the Quillen filtration is contained in the reduction of the dual degree filtration modulo $2$. 
To do this we use the prefix chain description of the dual degree filtration from Proposition \ref{prop:prefixdualdegree}. 
Then we show that the dual degree filtration is contained in the Kalinin filtration.

Therefore, we find the following sequence of inclusions of the filtrations of the $\Ztwo$-tope spaces
\[\Qcal_{\bullet}(M) \subseteq \overline{\Pcal}_{\bullet}(M) \subseteq \Kcal_{\bullet}(M).\]
We then compare the dimensions of $\Qcal_p(M)$ and $\Kcal_p(M)$ and conclude that all three filtrations are equal. 

\subsection{The Dual VG Filtration Contains the Quillen Filtration}

We show that prefix chains are sufficient to span the generators of the Quillen filtration.
As an intermediate step, we consider affine spaces $U$ of $\Tcal_\Fcal$ that are parallel to subspaces spanned by a subset of the basis $\Bcal=\{\dd_{\Fcal,1},\ldots,\dd_{\Fcal,d}\}$, which we dub as the {\em affine coordinate subspaces} of $\Tcal_\Fcal$; they correspond to {\em affine coordinate chains} $\gamma_{U,\Bcal,\vv}$.

Whenever an affine coordinate subspace or chain is being considered, the flag of flats it is associated with would be clear and we always work with the basis consisting of $\dd$'s, so for brevity, we do not define the bases in the proofs.

Since some of the results here are also applied in Section~\ref{sec:cosheaf}, we work over $\ZZ$ when dealing with the dual degree filtration.

\begin{proposition} \label{prop:prefix_to_coor}
Let $\Fcal:\emptyset=F_0 \subsetneq  F_1 \subsetneq  \ldots \subsetneq  F_d=E$ be a complete flag of flats of $M$.
Then for any $p$-dimensional affine coordinate subspace $U$ of $\Tcal_\Fcal$ and $\vv\in U$, $\gamma_{U,\Bcal,\vv}$ is contained in $\Pcal_p(M)$.
\end{proposition}

\begin{proof}
By the definition of $\Pcal_p(M)$ using Heaviside functions, it suffices to show that for any $e_1,\ldots,e_{p-1}\in E$, $(h_{e_1}\ldots h_{e_{p-1}})(\gamma_{U,\Bcal,\vv})=0$.
Suppose $U$ is parallel to $\langle\dd_{i_1},\ldots,\dd_{i_p} \rangle$.
Then there exists $1\leq j\leq p$ such that $\dd_{i_j}$ does not contain any of $e_i$'s, pairing up the topes in $U$ as $\{\uu,\uu+\dd_{i_j}\}$'s, evaluating the two topes in each pair by $h_{e_1}\ldots h_{e_{p-1}}$ produces equal value, but the two topes are of opposite signs in $\gamma_{U,\Bcal,\vv}$.
\end{proof}

We now work with $\Ztwo$-coefficients and we can identity an affine subspace $U$ with a chain $\gamma_U$ without specifying the original $\vv$. Our next step is to get all affine subspaces from affine coordinate subspaces.

\begin{proposition} \label{prop:coor_to_all}
Let $V$ be a $d$-dimensional vector space over $\FF_2$ with standard basis $\ee_1, \dots, \ee_d$.
Then the $p$-th piece $\Ical_p(V)$ of the Quillen filtration of $V$ is contained in the subspace generated by $\gamma_U$, ranging over all affine coordinate subspaces $U$ of dimension $p$.
\end{proposition}

\begin{proof}
By Proposition~\ref{prop:SS_generate_Quillen}, it suffices to show that all elements of the form $\gamma_{V'}$, where $V'$ is a $p$-dimensional subspace of $V$, can be written as the linear combination of $\gamma_U$'s with $U$'s being affine coordinate subspaces.
For a subset $A\subset V$, define the weight of $A$ to be the sum of weights of the elements in $A$ with respect to the standard basis.
For a subspace $V'$ of $V$, define the weight of $V'$ to be the minimum weight of $\mathcal{B}$ among all bases $\mathcal{B}$ of $V'$.
We prove the statement by an induction on the weight of the subspace $V'$.
If the weight of $V'$ is $p$, then necessarily $V'$ has a basis consisting of standard basis elements, hence $V'$ is a coordinate subspace.

Now suppose the weight of $V'$ is larger than $p$, pick a basis $\vv_1,\ldots,\vv_p$ of $V'$ of minimum weight.
Without loss of generality, we can assume that $\vv_p$ is not a standard basis element, and write $\vv_p=\vv' + \ee_k$ where $\ee_k$ is not involved when $\vv'$ is written as sum of $\ee_i$'s.
Note that $\vv',\ee_k\not\in\langle\vv_1,\ldots,\vv_{p-1} \rangle$, for otherwise $\vv_1,\ldots,\vv_{p-1},\ee_k$ (respectively $\vv_1,\ldots,\vv_{p-1},\vv'$) is a basis of smaller weight.
Let $W'=\langle\vv_1,\ldots,\vv_{p-1},\vv' \rangle$, and $W''=\langle\vv_1,\ldots,\vv_{p-1},\ee_k \rangle$.
By induction hypothesis, $\gamma_{W'}$ and $\gamma_{W''}$ can each be written as a linear combination of the elements corresponding to affine coordinate subspaces.
By applying a translation by $\ee_k$ to the summands in the first linear combination, $\gamma_{\ee_k+W'}$ can also be written as a linear combination of the elements corresponding affine coordinate subspaces.
Since $\ee_k+W' = (\ee_k+\langle\vv_1,\ldots,\vv_{p-1} \rangle) \sqcup (\vv_p+\langle\vv_1,\ldots,\vv_{p-1} \rangle)$ and $G''=\langle\vv_1,\ldots,\vv_{p-1} \rangle\sqcup (\ee_k+\langle\vv_1,\ldots,\vv_{p-1} \rangle)$, while $V'=\langle\vv_1,\ldots,\vv_{p-1} \rangle\sqcup(\vv_p+\langle\vv_1,\ldots,\vv_{p-1} \rangle)$, we have that $\gamma_{V'}=\gamma_{\ee_k+W'}+\gamma_{W''}$ is a linear combination of affine coordinate subspaces.
\end{proof}

We have established the following corollary. 

\begin{corollary} \label{coro:Quillen_in_prefix}
Let $M$ be an oriented matroid.
Then for every $p$, $\Qcal_p(M)\subset\overline{\Pcal}_p(M)$.
\end{corollary}

\begin{proof}
It follows from Proposition \ref{prop:SS_generate_Quillen} that $\Qcal_p(M)$ is generated by elements of the form $\gamma_U$ for some $p$-dimensional affine subspaces $U$. 
From Proposition \ref{prop:prefix_to_coor}, we obtain that $\gamma_U$ is contained in $\overline{\Pcal}_p(M)$, where $U$ is any affine coordinate subspace of $\Tcal_{\Fcal}$ for a complete flag of flats $\Fcal$. Therefore, the statement follows.
\end{proof}

\subsection{The Kalinin Filtration Contains the Dual VG Filtration}
\label{sec:Kalinin_prefix}

We prove that every $p$-th prefix chain is contained in the $p$-th part of the Kalinin filtration by exhibiting a sequence of $\beta_i$'s as in Definition~\ref{def:Kalinin_Z2Z}.
We work with $\Ztwo$ throughout this section unless otherwise specified.

\begin{theorem} \label{thm:Kalinin_prefix}
Let $\Fcal:\emptyset=F_0 \subsetneq  F_1 \subsetneq \ldots \subsetneq F_d=E$ be a complete flag of flats and let $T_{\vv}\in\Tcal_\Fcal$.
Then $\gamma_{\Fcal,\vv,p}\in\Kcal_p(M)$.
\end{theorem}

\begin{proof}
For brevity of the notation, we assume (by reorientation) $\vv={\bf 0}$ whenever only one specific prefix chain is being considered.

For the case $p=1$, we have $\gamma_{\Fcal, {\bf 0},1}=T_{{\bf 0}}+T_{\dd_1}$, while $Z(T_{\bf 0}\setminus F_1,T_{\bf 0})$ is a $1$-dimensional cell whose boundary is precisely the two vertices $Z(T_{\bf 0},T_{\bf 0})$ and 
$Z(T_{\dd_1},T_{\bf 0})$ .
So we can take $\beta_1$ to be $Z(T_{\bf 0}\setminus F_1,T_{\bf 0})$.

Now suppose the statement is true for $p$. More precisely, for every $\gamma_{\Fcal,\vv,p}$ we suppose for $i \leq p$ that 
there exists $ \beta_i \in C_i(\Sal_M; \Ztwo)$ such that 
$\partial\beta_1= i_{\ast} \gamma$ and $\partial\beta_i=\beta_{i-1}+\overline{\beta_{i-1}}$ for $ 2 \leq i \leq p$. 
We further assume that we can choose 
\begin{equation}\label{eq:beta_p}
    \beta_{p}  = \sum_{{\bf u}\in U_{\Fcal,\vv,p-1}} Z(T_{\vv}\setminus F_{p},T_{\bf u}).
\end{equation}

Notice that this holds in the case $p = 1$ above.

We now pass to the case of $p+1$, and prove that $\gamma_{\Fcal,{\bf 0},p+1}$
is in $\Kcal_{p+1}(M)$ by finding $i$-chains $\beta_i$ for $i = 1, \dots, p+1$. Moreover, we have 
$$\beta_{p+1} = \sum_{{\bf u}\in U_{\Fcal, {\bf 0}, p}} Z(T_{\bf 0}\setminus F_{p+1},T_{\bf u}).$$
 
We start with noting $$\gamma_{\Fcal, {\bf 0},p+1}=\gamma_{\Fcal, {\bf 0},p}+\gamma_{\Fcal, \dd_{p+1},p}.$$ 
By induction, for $i = 1, \ldots, p$, 
there are $i$-chains $\beta_i'$ and $\beta_i''$ certifying that $\gamma_{\Fcal, {\bf 0},p}$ and $\gamma_{\Fcal, \dd_{p+1},p} $ are in $\Kcal_p$.
For $i \leq p$, we can take $\beta_i = \beta_i' + \beta_i''$. Using the induction assumption, it is easily verified that $\partial \beta_i = \beta_{i-1} +\overline{\beta_{i-1}}$ for $i = 1, \ldots, p$. 
It remains to show that $\partial \beta_{p+1} = \beta_p + \overline{\beta_{p}}$, where $\beta_p = \beta_p' + \beta_p''$.

We spend some time simplifying the conjugation of $\beta'_p$. 
Observe that
$$\overline{Z(T_{\bf 0}\setminus F_p,T_{\bf u})}
=Z(T_{\bf 0}\setminus F_p,-T_{\bf u})
=Z(T_{\bf 0}\setminus F_p,T_{\ff_d + \bf u})
=Z(T_{\bf 0}\setminus F_p,T_{\ff_p+ {\bf u}}),$$ 
since $(T_{\bf 0}\setminus F_p) \circ T_{\ff_d + \bf u} = (T_{\bf 0}\setminus F_p) \circ T_{\ff_p+ {\bf u}}$.
Hence we have
 $$\overline{\beta'_p} = \sum_{{\bf u}\in U_{\Fcal,\ff_p,p-1}} Z(T_{\bf 0}\setminus F_p,T_{\bf u}).$$

By the same argument, we also have 
$\gamma_{\Fcal, \dd_{p+1},p}\in\Kcal_p(M)$.
More precisely, $\gamma_{\Fcal, \dd_{p+1},p}$ can be obtained from $\gamma_{\Fcal, {\bf 0},p}$ by replacing every $T_{\bf v}$ by $T_{{\bf v}+\dd_{p+1}}$ in the sum.
By induction, when constructing $\beta''_{p'}, p'<p$ for $\gamma_{\Fcal, \dd_{p+1},p}$, one can simply apply the analogous replacement of $Z(T_{\bf 0}\setminus F_{p'},T_\vv)$ by $Z(T_{\dd_{p+1}}\setminus F_{p'},T_{{\bf v}+\dd_{p+1}})$.

In particular, we can choose the corresponding $p$-chain to be 
$$\beta_p'' = \sum_{{\bf u}\in U_{\Fcal,\dd_{p+1}, p-1}} Z(T_{\dd_{p+1}} \setminus F_p,T_{\bf u}).$$

Nevertheless, we note that $Z(T_{\dd_{p+1}} \setminus F_p,T_{\uu+\dd_{p+1}})=Z(T_{\dd_{p+1}} \setminus F_p,T_{\bf u})$ for ${\bf u}\in U_{\Fcal,{\bf 0},p-1}$, as $$(T_{\dd_{p+1}}\setminus F_p)\circ T_{{\bf u}+\dd_{p+1}}=(T_{\dd_{p+1}}\setminus F_p)\circ T_{\bf u},$$ so the sum can actually be rewritten as
\begin{equation} \label{eq:2nd_chain}
  \beta_p'' =  \sum_{{\bf u}\in U_{\Fcal,{\bf 0}, p-1}} Z(T_{\dd_{p+1}} \setminus F_p,T_{\bf u}).
\end{equation}
Analogous to the above, 
we have $$\overline{\beta''_p} = \sum_{{\bf u}\in U_{\Fcal,\ff_p, p-1}} Z(T_{\dd_{p+1}} \setminus F_p,T_{\bf u}).$$

Hence, it remains to show that the boundary of
$\beta_{p+1}$ is equal to 
\begin{equation} \label{eq:target_chain}
    \sum_{{\bf u}\in U_{\Fcal,{\bf 0}, p}}\big[Z(T_{\bf 0}\setminus F_p,T_{\bf u})+Z(T_{\dd_{p+1}} \setminus F_p,T_{\bf u})\big],
\end{equation}
here we use the simple fact that $U_{\Fcal,{\bf 0}, p}=U_{\Fcal,{\bf 0}, p-1}\sqcup U_{\Fcal,\ff_p, p-1}$.

Every $p$-cell on the boundary of $\beta_{p+1}$ is of the form $Z(T\setminus F,T_{\bf u})$ for some flat $F\subset F_{p+1}$ of rank $p$, some tope $T$ that agrees with $T_{\bf 0}$ outside of $F_{p+1}$, and some ${\bf u}\in U_{\Fcal, {\bf 0}, p}$.

Case I: $F\neq F_p$.\\
There exists some $1\leq k\leq p$ such that $F\cap(F_k\setminus F_{k-1})=\emptyset$, for otherwise we would have $F_p = \langle F\cap(F_i\setminus F_{i-1}): i=1,\ldots,p\rangle\subsetneq F$, which is impossible by considering rank.
For any ${\bf u}\in U_{\Fcal,{\bf v}_0, p}$, we have $Z(T\setminus F,T_{\bf u})=Z(T\setminus F,T_{{\bf u}+\dd_k})$ because $(T\setminus F)\circ T_{\bf u}=(T\setminus F)\circ T_{{\bf u}+\dd_k}$.
By pairing up the elements of $U_{\Fcal,{\bf 0}, p}$ into $\{{\bf u}, {\bf u}+\dd_k\}$'s, we can see that the boundary does not contain cells of the form $Z(T\setminus F,T_{\bf u})$ for $F\neq F_p$. 

Case II: $F=F_p$.\\
In such a case, a tope $T$ restricted to $F_{p+1}\setminus F_p$ must be equal to that of $T_{\bf 0}$ or $T_{\dd_{p+1}}$, for otherwise, $T|_{F_{p+1}}/F_p$ would be a third tope of the rank 1 matroid $M|_{F_{p+1}}/F_p$.
For every ${\bf u}\in U_{\Fcal,{\bf 0}, p}$, the cell $Z(T_{\bf 0}\setminus F_{p+1},T_{\bf u})$ is the unique cell in $\beta_{p+1}$ whose boundary contains $Z(T_{\bf 0}\setminus F_p,T_{\bf u})$; it is also the unique cell in $\beta_{p+1}$ which contains $Z(T_{\dd_{p+1}}\setminus F_p,T_{\bf u})$ in its boundary: for any other ${\bf u}'\in U_{\Fcal,{\bf 0}, p}$, we have $(T_{\bf 0}\setminus F_p)\circ T_{\bf u}\neq (T_{\bf 0}\setminus F_p)\circ T_{{\bf u}'}$, so $Z(T_{\bf 0}\setminus F_p,T_{\bf u})\neq Z(T_{\bf 0}\setminus F_p,T_{{\bf u}'})$, and an analogous argument for $Z(T_{\dd_{p+1}}\setminus F_p,T_{\bf u})$.
Hence the two cells $Z(T_{\bf 0}\setminus F_p,T_{\bf u}),Z(T_{\dd_{p+1}}\setminus F_p,T_{\bf u})$ appear in the boundary exactly once each, as claimed in (\ref{eq:target_chain}).
\end{proof}

It turns out that the special chains $\beta_p$'s constructed in the above proof have been studied by Denham with $\ZZ$-coefficients. In the following definition, we fix an orientation of the order complex $\Delta(\Lcal(M|_{F_p}))$ and make use of the canonical homeomorphism between it and each $Z(T_{\vv}\setminus F_p,T_{\bf u})$ to give the latter an orientation.

\begin{definition}\label{def:pbricks} \cite[Definition~2.4]{Denham_OS}
Given the datum $(\Fcal,\vv,p)$ to define a prefix chain as well as an arbitrary orientation of $\Delta(\Lcal(M|_{F_p}))$, the corresponding {\em $p$-brick} is 
$$\alpha_{\Fcal,\vv,p} = \sum_{{\bf u}\in U_{\Fcal,\vv, p}} (-1)^{w({\bf u})} Z(T_{\vv}\setminus F_p,T_{\bf u})\in C_p(\Sal_M;\ZZ).$$
\end{definition}

\begin{example} \label{ex:Kalinin} \rm
Following with the on-going Example \ref{ex:SalvettiU23}, the $1$-brick and $2$-brick associated to $\gamma_{\Fcal,T_1,1}, \gamma_{\Fcal,T_1,2}$ are $\alpha_{\Fcal,T_1,1}=Z(\zeta_1,T_1)-Z(\zeta_1,T_2)$ and $\alpha_{\Fcal,T_1,2}=Z(O,T_1)-Z(O,T_2)+Z(O,T_4)-Z(O,T_5)$, respectively.
\end{example}

\begin{remark} \label{rm:Kalinin}
\normalfont
(1) When the oriented matroid is realisable by a real hyperplane arrangement, the Salvetti complex can be realised as a strong deformation retract of the complement of its complexification, whose conjugation as a cellular complex coincides with the restriction of the conjugation of the ambient complex Euclidean space.
Hence the construction of $p$-bricks remains valid if we replace the role of $\Sal_M$ by the complement of its complexification.

(2) With the convention of orientations of the cells in $\Sal_M$, it is possible to verify that integral $p$-bricks can be recursively constructed as in the above proof (which we omitted for sake of brevity). However, working with $\ZZ$-coefficients, the difference of two choices of $\beta_i$ does not cancel out upon adding to its conjugation, leaving ambiguity. So a direct extension extension of Kalinin filtration with $\ZZ$-coefficients using $p$-bricks is not well-defined after all.
\end{remark}

\begin{corollary} \label{coro:Prefix_in_Kalinin}
Let $M$ be an oriented matroid.
Then for every $p$, $\overline{\Pcal}_p(M)\subset\Kcal_p(M)$.
\end{corollary}

\begin{proof}[Proof of Theorem \ref{mainthm:Z2}]
By Corollary~\ref{coro:Quillen_in_prefix} and \ref{coro:Prefix_in_Kalinin}, we have $\Qcal_p(M)\subset\overline{\Pcal}_p(M)\subset\Kcal_p(M)$.
By comparing the dimensions of the $\Qcal_p(M)$ and $\Kcal_p(M)$ using Proposition~\ref{prop:Kalinin}, Theorem~\ref{thm:sign_cosheaf_dim}, and Theorem~\ref{thm:H_eq_OS}, we conclude that all inclusions are actually equalities.
\end{proof}

\section{Proof of Theorem \ref{mainthm:maps}}

We begin by recalling the following result of Denham which describes an isomorphism between the associated gradeds of the dual Varchenko--Gelfand filtration and the homology of the Salvetti complex in terms of prefix chains, upon choosing an arbitrary orientation for every $\Delta(\Lcal(M|_p))$.
The isomorphism comes from analysing a combinatorially defined filtration of the Salvetti complex that leads to a filtration of its chain complexes with $\ZZ$-coefficients and a spectral sequence. 

\begin{theorem} \label{thm:B1} \cite[Theorem~3.5]{Denham_OS}
There is a well-defined map $$\abv_p:\Pcal_p(M)\rightarrow H_p(\Sal_M;\ZZ)$$ given by linearly extending $\gamma_{\Fcal,\vv,p}\mapsto [\alpha_{\Fcal,\vv,p}]$, which descends into a well-defined isomorphism $$\abv_p:\Pcal_p(M)/\Pcal_{p+1}(M)\rightarrow H_p(\Sal_M;\ZZ).$$
\end{theorem}

\begin{proposition} \label{prop:a_eq_bv}
    As maps from $\Kcal_p(M)=\overline{\Pcal}_p(M)$  to $H_p(\Sal_M;\Ztwo)$ $bv_p$ and $\overline{\abv}_p$ are equal.
\end{proposition}

\begin{proof}
It suffices to compare the two maps over the generators $\gamma_{\Fcal,\vv,p}$'s. By definition,  $bv_p(\gamma_{\Fcal,\vv,p})=[\beta_p+\overline{\beta_p}]$ where $\beta_p$ is described in (\ref{eq:beta_p}) as $\sum_{{\bf u}\in U_{\Fcal,\vv,p-1}} Z(T_{\vv}\setminus F_{p},T_{\bf u})$, and $\overline{Z(T_{\vv}\setminus F_{p},T_{\bf u})}=Z(T_{\vv}\setminus F_{p},-T_{\bf u})=Z(T_{\vv}\setminus F_{p},T_{\ff_p+\bf u})$. As ${\bf u}$'s range over $U_{\Fcal,\vv,p-1}$, $\ff_p+{\bf u}$'s range over $U_{\Fcal,\ff_p+\vv,p-1}$, and $U_{\Fcal,\vv, p}=U_{\Fcal,\vv, p-1}\sqcup U_{\Fcal,\ff_p+\vv, p-1}$, hence $\beta_p+\overline{\beta_p}$ is precisely $\alpha_{\Fcal,\vv, p}$ in Definition~\ref{def:pbricks} over $\Ztwo$.
\end{proof}

To prove the rest of Theorem~\ref{mainthm:maps}, we utilise the explicit isomorphism between the cellular cohomology of $\Sal_M$ and Orlik--Solomon algebra described by Bj\"orner and Ziegler \cite{BZ92}. We then show that under the dual of this isomorphism, the above map provided by Denham, which is equal to the Viro homomorphisms from Kalinin filtration as we have shown in the last section, is equivalent to the maps $\qbv_p$ defined first in \cite{RS}.

We work with $\Ztwo$-coefficients for Theorem~\ref{mainthm:maps} as some notions such as Quillen filtration and Kalinin are only defined over $\Ztwo$. But we work with $\ZZ$-coefficients in this section all the way up till the very proof of Theorem~\ref{mainthm:maps} since the respective results over $\ZZ$ might be of independent interest. 

To prove the theorem, we use explicit bases of $\OS^p(\underline{M})$. 
Fix an arbitrary ordering of $E$.
A {\em broken circuit} is obtained by removing the minimum element from a circuit of $\underline{M}$.
A {\em NBC-set} is a subset of $E$ that does not contain any broken circuit; it is necessarily independent. For each NBC-set $S=(i_1>\ldots>i_p)$, we can associate an element $e^*_{i_1}\wedge\ldots\wedge e^*_{i_p}$ in the Orlik--Solomon algebra of $\underline{M}$.

\begin{theorem} \cite[Theorem~7.10.2]{Bjo92}
For any ordering of $E$, the elements of $\OS(\underline{M};\ZZ)$ corresponding to NBC-sets of size $p$ form a $\ZZ$-basis of $\OS^p(\underline{M}; \ZZ)$, hence they also descend into a $\Ztwo$-basis of $\OS^p(\underline{M}; \Ztwo)$.
\end{theorem}

Denote by $\BC^p(\underline{M})$ the collection of NBC-sets of size $p$, so $|\BC^p(\underline{M})|=b_p(\underline{M})$.
Let $S:=\{i_1,\ldots,i_p\}$ be a NBC-set, with $i_1>\ldots>i_p$ with respect to the ambient ordering.
Consider the partial flag $\Fcal_S:\emptyset=F_0 \subsetneq F_1=\langle i_1\rangle \subsetneq  F_2=\langle i_1,i_2\rangle \subsetneq \ldots \subsetneq  F_p=\langle i_1,\ldots,i_p\rangle$.
A partial flag (and extension thereof) obtained this way is a {\em NBC-flag}; the corresponding prefix chains are {\em NBC-chains}.
We state an elementary fact about NBC-sets and NBC-chains.

\begin{lemma} \label{lem:unique_NBC}
Let $S,\Fcal_S$ be as above, and let $S'\neq S$ be another NBC-set of size $p$.
Then $S'\cap(F_k\setminus F_{k-1})=\emptyset$ for some $k\leq p$.
\end{lemma}

\begin{proof}
Suppose $S'$ is a counterexample.
Then $|S'\cap(F_k\setminus F_{k-1})|=1$ for every $k\leq p$.
Let $k$ be the smallest index such that $S\cap(F_k\setminus F_{k-1})\neq S'\cap(F_k\setminus F_{k-1})=\{i'\}$.
The size $k+1$ set $\{i_1,\ldots,i_k,i'\}\subset F_k$ necessarily contains a circuit $C$, moreover, $i',i_k\in C$.
If $i'<i_k$, then $C\setminus\{i'\}\subset S$ is a broken circuit, otherwise $i_k$ must be the smallest element in $C$, and $C\setminus\{i_k\}\subset S'$ is a broken circuit.
\end{proof}

Next, we state the cellular cohomological description of the elements of the Orlik--Solomon algebra (which are ``de Rham'' in nature) due to Bj\"{o}rner--Ziegler.

\begin{theorem} \cite[Theorem 7.2]{BZ92} \label{thm:BZ_Zbasis}
Fix an arbitrary ordering of the ground set $E$, and for each NBC-set of size $p$, fix an arbitrary ordering of its element (not necessarily related to the ambient ordering of $E$).
Then $\{c^S: S\in \BC^p(\underline{M})\}$ is a $\ZZ$-basis of $H^p(\Sal_M;\ZZ)$, where for any ordered NBC-set $S=(i_1,\ldots,i_p)$ and any positively oriented $p$-simplex $\Delta$ of the {\em fine} Salvetti complex whose vertices are $$w(L_0,T_0)<w(L_1,T_1)<\ldots<w(L_p,T_p),$$ $c^S(\Delta)$ equals $1$ whenever 
\begin{enumerate}
    \item for $0\leq s<t\leq p$, $(L_s)_{i_t}=+$; and
    \item for $1\leq t\leq s\leq p$, $(L_s)_{i_t}=0$ and $(T_s)_{i_t}=+$,
\end{enumerate}
otherwise $c^S(\Delta)=0$.

Moreover, the map $BZ$ given by extending $c^S\mapsto e^*_{i_1}\wedge\ldots\wedge e^*_{i_p}$ linearly is an isomorphism between the cellular cohomology $H^p(\Sal_M;\ZZ)$ and $\OS^p(\underline{M};\ZZ)$.
\end{theorem}

\begin{proposition} \label{prop:alpha_dual_basis}
Let $S$ be a NBC-set of size $p$ and let $\tilde{\Fcal}_S$ be an arbitrary complete flag extending $\Fcal_S:F_0 \subsetneq  \ldots \subsetneq F_p$, also pick an arbitrary $T_{\vv}\in\Tcal_{\tilde{\Fcal}_S}$ as origin.
Then $c^{S}(\alpha_{\tilde{\Fcal}_S,\vv,p})=\pm 1$ (the precise sign depends on the orientation of $\alpha_{\tilde{\Fcal}_S,\vv,p}$ and ordering of elements of $S$), and $c^{S'}(\alpha_{\tilde{\Fcal}_S,\vv,p})=0$ for every other NBC-set $S'$.
\end{proposition}

\begin{proof}
Recall that $\alpha_{\tilde{\Fcal}_S,\vv,p}=\sum_{{\bf u}\in U_{\tilde{\Fcal}_S,{\bf v}, p}} (-1)^{w({\bf u})} Z(T_{\vv}\setminus F_p,T_{\bf u})$.
The $p$-dimensional simplices contained by $Z(T_{\vv}\setminus F_p,T_{\bf u})$ in the fine Salvetti complex $\widetilde{\Sal}_M$ are in one to one correspondence with the conformal flags of covectors $\Tcal\ni L_0>\ldots>L_p=T_{\vv}\setminus F_p$, namely, each such flag corresponds to the simplex $\Delta_{T_{\bf u}}^{L_0>\ldots>L_p}$ whose vertices are $w(L_0,L_0\circ T_{\bf u})<\ldots<w(L_p, L_p\circ T_{\bf u})$.

Suppose the elements of $S$ are $i_1>\ldots>i_p$ with respect to the ambient ordering.
We use the ordering $(i_1,\ldots,i_p)$ for $c^S$.
Suppose $c^S(\Delta_{T_{\bf u}}^{L_0>\ldots>L_p})\neq 0$.
The first part of Condition (2) in Theorem~\ref{thm:BZ_Zbasis} implies that $(L_s)|_{i_t}=0,\forall t\leq s$, so the complement of $\supp(L_s)$ must be $\langle i_1,\ldots,i_s\rangle=F_s$.
The second part of Condition (2) then uniquely specifies a tope $T_{\bf u}$ with ${\bf u}\in U_{\tilde{\Fcal}_S,\vv,p}$, since each condition $(T_{\bf u})_{i_k}=+$ specifies the value of $T_{\uu}$ over $F_k\setminus F_{k-1}$.
Next, as $L_0$ agrees with $T_{\vv}$ outside of $F_p$, $L_0$ must be a tope in $\Tcal_{\tilde{\Fcal}_S}$: for $k\leq p$, $L_0\setminus F_k=L_k\in\Lcal$, and for $k>p$, $L_0\setminus F_k=T_{\vv}\setminus F_k\in\Lcal$.
So similarly, Condition (1) uniquely specifies a tope $L_0\in\Tcal_\Fcal$.
Summarising, there is a unique simplex in $\alpha_{\tilde{\Fcal}_S,\vv,p}$ whose evaluation of $c^S$ is non-zero (and is $\pm 1$).

Now suppose $S'\neq S$ is another NBC-set of size $p$.
By Lemma~\ref{lem:unique_NBC}, there exists some $1\leq k\leq p$ such that $S'\cap (F_k\setminus F_{k-1})=\emptyset$.
Hence in $\alpha_{\tilde{\Fcal}_S,\vv,p}$, we have $c^{S'}(\Delta_{T_{\uu}}^{L_0>\ldots>L_p})=c^{S'}(\Delta_{T_{\uu+\dd_k}}^{L_0>\ldots>L_p})$, as the values of the topes over $F_k\setminus F_{k-1}$ are irrelevant with the conditions in Theorem~\ref{thm:BZ_Zbasis}. The two simplices are respectively contained in $Z(T_{\vv}\setminus F_p,T_{\uu})$ and $Z(T_{\vv}\setminus F_p,T_{\uu+\dd_k})$, in which our convention of orientation via canonical homeomorphism to $\Delta(\Lcal(M|_{F_p}))$ induces the same orientation of them.
However, $Z(T_{\vv}\setminus F_p,T_{\uu})$ and $Z(T_{\vv}\setminus F_p,T_{\uu+\dd_k})$ are of opposite signs in $\alpha_{\tilde{\Fcal}_S,\vv,p}$, hence the evaluations of $c^{S'}$ on $\Delta_{T_{\uu}}^{L_0>\ldots>L_p}$ and $\Delta_{T_{\uu+\dd_k}}^{L_0>\ldots>L_p}$ cancel each other out.
By pairing up the elements of $U_{\tilde{\Fcal}_S,\vv,p}$ into $\{{\bf u}, {\bf u}+\dd_k\}$'s, we have $c^{S'}(\alpha_{\tilde{\Fcal}_S,\vv,p})=0$.
\end{proof}

\begin{corollary} \label{coro:NBC_Hp_basis}
For every NBC-set $S$ of size $p$, extend the NBC-chain of $S$ arbitrarily into a complete flag $\tilde{\Fcal}_S$ and pick an $T\in\Tcal_{\tilde{\Fcal}_S}$.
Then the collection of $p$-bricks $\alpha_{\tilde{\Fcal}_S,T,p}$ (or more precisely, the homology classes they represent) form a $\ZZ$-basis of $H_p(\Sal_M;\ZZ)$.
\end{corollary}

\begin{proof}
By the universal coefficient theorem and the fact that all homology and cohomology groups of $\Sal_M$ are torsion free, $H_p(\Sal_M;\ZZ)$ is canonically isomorphic to the dual of $H^p(\Sal_M;\ZZ)$.
By Theorem~\ref{thm:BZ_Zbasis}, $c^S$'s form a $\ZZ$-basis of $H^p(\Sal_M;\ZZ)$, and by the conventions of Proposition~\ref{prop:alpha_dual_basis}, the dual elements with respect to this basis are, up to signs, the $\alpha_{\tilde{\Fcal}_S,T,p}$'s. Hence they form a $\ZZ$-basis of $H_p(\Sal_M;\ZZ)$.
\end{proof}

Since we work with $\Ztwo$-coefficients in the theorem, the isomorphism $BZ$ is independent of the ordering and orientation convention.
Dualising the isomorphism $BZ$, we obtain $BZ^\vee:H_p(\Sal_M;\Ztwo)\xrightarrow{\cong}\OS_p(\underline{M};\Ztwo)$.
Hence, we can compare the map $\overline{\abv}_p:\overline{\Pcal_p(M)}\rightarrow H_p(\Sal_M;\Ztwo)$ described in Theorem~\ref{thm:B1} and the map $\qbv_p:\Qcal_p(M) \rightarrow \OS_p(\underline{M};\Ztwo)$ as described in \cite[Definition 4.9]{RS} in Equation \ref{eq:defqp} as follows:
\begin{equation} \label{eq:bv_eq_bv}
\begin{CD}
\overline{\Pcal}_p(M)    @>\overline{\abv}_p>>  H_p(\Sal_M;\Ztwo)\\
@V\parallel VV        @V BZ^\vee VV\\
\Qcal_p(M)      @>\qbv_p>>  \OS_p(\underline{M};\Ztwo)
\end{CD}
\end{equation}

\begin{theorem} \label{thm:B2}
The diagram (\ref{eq:bv_eq_bv}) commutes.
\end{theorem}

\begin{proof}
It suffices to verify the commutativity for prefix chains on the left hand side.
Let $\gamma_{\Fcal,\vv,p}$ be a prefix chain with $\Fcal:F_0\subsetneq \ldots \subsetneq F_d$, in which we may assume that it is a NBC-chain associated with $S\in\BC^p(\underline{M})$ by choosing a suitable ordering of $E$, so $S=\{i_1,\ldots,i_p\}$ is given by $i_j=\min F_j\setminus F_{j-1}$.
By Corollary~\ref{coro:NBC_Hp_basis}, the image $\alpha_{\Fcal,\vv,p}$ of the chain is the dual element, indexed by $S$, with respect to the basis $\{c^S:S\in\BC^p(\underline{M})\}$ of $H^p(\Sal_M;\Ztwo)$.

We claim that under $BZ^\vee$, the corresponding element in $\OS_p(\underline{M};\Ztwo)$ is $e_{F_1}\wedge e_{F_2\setminus F_1}\wedge\ldots\wedge e_{F_p\setminus F_{p-1}}$.
The element in $\OS^p(\underline{M};\Ztwo)$ corresponding to $c^S$ is $e^*_{i_1}\wedge\ldots\wedge e^*_{i_p}$, whose evaluation on $e_{F_1}\wedge e_{F_2\setminus F_1}\wedge\ldots\wedge e_{F_p\setminus F_{p-1}}$ equals $\det(e^*_{i_j}(e_{F_k\setminus F_{k-1}}))=\det(\delta_{j,k})=1$.
For any other NBC-set $S'=\{i'_1,\ldots,i'_p\}\neq S$ of size $p$, by Lemma~\ref{lem:unique_NBC}, there exists $k\leq p$ such that $S'\cap(F_k\setminus F_{k-1})=\emptyset$, so $e^*_{i'_j}(e_{F_k\setminus F_{k-1}})=0,\forall j$, which implies $e^*_{i'_1}\wedge\ldots\wedge e^*_{i'_p}$ evaluates to zero on $e_{F_1}\wedge e_{F_2\setminus F_1}\wedge\ldots\wedge e_{F_p\setminus F_{p-1}}$. This proves the claim. 

Finally, consider the image of a prefix chain $\gamma_{\Fcal,\vv,p}$ under $\qbv_p$. 
By the translation of $\Tcal_{\Fcal}$ by $\vv$, we obtain the subspace $V$.
The image of $e_{F_1}\wedge e_{F_2\setminus F_1}\wedge\ldots\wedge e_{F_p\setminus F_{p-1}}\in \bigwedge^p V$ under the local isomorphism $\bigwedge^p V\cong\Ical_p(V)/\Ical_{p+1}(V)$ is $[\sum_{\uu\in\langle \dd_1,\ldots,\dd_p\rangle} T_{\uu}]=[\gamma_{\Fcal,{\bf 0},p}]$ (see Section~\ref{sec:Quillen}).
Re-translate back by $\vv$ shows that $e_{F_1}\wedge e_{F_2\setminus F_1}\wedge\ldots\wedge e_{F_p\setminus F_{p-1}}$ is the image of $\gamma_{\Fcal,\vv,p}$ under $\qbv_p$ as wanted. Thus the diagram commutes. 
\end{proof}

\begin{proof}[Proof of Theorem \ref{mainthm:maps}]
By Theorem \ref{mainthm:Z2} we have that $\overline{\Pcal}_p$, $\Kcal_p$, and $\Qcal_p$ are all the same subspace of the $\Ztwo$-tope space.
The equality of the maps $bv_p$ and $\overline{a}_p$ is Proposition~\ref{prop:a_eq_bv}, and the equality of these two maps and $\qbv_p$ up to $BZ^\vee$ is Theorem~\ref{thm:B2}.
\end{proof}

\section{A $\ZZ$-coefficients Filtration of Sign Cosheaf on a Matroid Fan} \label{sec:cosheaf}

In this section, we show that the dual Varchenko--Gelfand filtration of the $\ZZ$-tope space of an oriented matroid from Section \ref{sec:GV} can be made into a filtration of a $\ZZ$-variant of the sign cosheaf on the polyhedral fan of a matroid from \cite{RRS2}. We work over $\ZZ$ throughout this section unless otherwise specified. 

\subsection{Some Properties of Cordovil Algebra and its Dual} \label{sec:Cordovil}

We collect some properties of the Cordovil algebra, as well as formulate a notion of its dual algebra and prove some basic properties. For $S\subset E$, we denote $(x^*)^S:=\prod_{i\in S} x^*_i$ and similarly $x^S:=\prod_{i\in S} x_i$.

\begin{theorem} \cite[Corollary~2.5 and Corollary~2.8]{Cord02}
The element $(x^*)^S$ is nonzero in $\Acal^\bullet(M)$ if and only if $S$ is independent. Moreover, for any ordering of $E$, $(x^*)^S$'s, ranging over all NBC-sets $S\in\BC^p(\underline{M})$, form a $\ZZ$-basis of $\Acal^p(M)$. 
\end{theorem}

We think of $\Acal^\bullet(M)$ as the quotient of $\mathbb{Z}[x_i^*:i\in E]/\langle (x^*_i)^2: i\in E\rangle$, which has a canonical dual algebra\footnote{``SF'' stands for ``square-free''.} ${\rm SFSym}(E):=\mathbb{Z}[x_i:i\in E]/\langle x_i^2: i\in E\rangle$ given by the pairing $\langle(x^*)^S, x^T\rangle=\delta_{S,T}$, extended linearly. As such, we can define the graded dual $\Acal_\bullet(M)$ of $\Acal^\bullet(M)$, which is a subalgebra of ${\rm SFSym}(E)$.

\begin{definition}
    Given a datum $(\Fcal,\vv,p)$ that defines a prefix chain, define
    $$\epsilon_{\Fcal,\vv,p}=\prod_{i=1}^p\big(\sum_{j\in F_i\setminus F_{i-1}} T_\vv(j)x_j\big) \in {\rm SFSym}(E).$$
\end{definition}

Via the obvious isomorphism $\bigwedge(\Ztwo)^E$ and ${\rm SFSym}(E)\otimes\Ztwo$ given by the identification $e_{i_1}\wedge\ldots\wedge e_{i_p}\leftrightarrow x_{i_1}\ldots x_{i_p}$, the element $\epsilon_{\Fcal,\vv,p}$ is equal to $\qbv_p(\gamma_{\Fcal,\vv,p})$ over $\Ztwo$ from the calculation in the proof of Theorem~\ref{thm:B2}. Also note that a monomial is in the support of $\epsilon_{\Fcal,\vv,p}$ if and only if the indices of its variables form a transversal of $F_1\setminus F_0,\ldots, F_p\setminus F_{p-1}$.

\begin{lemma}
    The element $\epsilon_{\Fcal,\vv,p}$ is in $\Acal_p(M)$ for every prefix chain datum $(\Fcal,\vv,p)$.
\end{lemma}

\begin{proof}
    We verify that $\epsilon_{\Fcal,\vv,p}$ is killed by any element of the form $$\partial C:=\sum_{k=1}^{p+1} C(i_k) x^*_{i_1}\ldots\widehat{x^*_{i_k}}\ldots x^*_{i_t}$$ from Definition~\ref{def:Cordovil}, where $C$ is a signed circuit of $M$ supported on $\{i_1,\ldots,i_{p+1}\}$. In order to simplify the notation, we assume $i_k=k$ for all $k$.
    The statement is trivial when $\{1,\ldots,p+1\}\cap (F_i\setminus F_{i-1})=\emptyset$ for some $i$, so without loss of generality we assume $i\in F_i\setminus F_{i-1}$ for every $i\leq p$, and $p+1\in F_p\setminus F_{p-1}$. The only two terms in $\partial C$ that do not vanish over $\epsilon_{\Fcal,\vv,p}$ are $C(p)x_1^*\ldots \widehat{x_p^*}x_{p+1}^*$ and $C(p+1)x_1^*\ldots x_p^*$, with the dual terms in $\epsilon_{\Fcal,\vv,p}$ having coefficients $T_\vv(1)\ldots \widehat{T_\vv(p)}T_\vv(p+1)$ and $T_\vv(1)\ldots T_\vv(p)$, respectively. 
    Summarizing, the evaluation is zero if $C(p)C(p+1)=-T(p)T(p+1)$, which follows from the orthogonality of signed circuits and signed cocircuits: consider the restriction of $C$ and $T$ as a signed circuit (supported on $\{p,p+1\}$) and a cocircuit (supported on $F_p\setminus F_{p-1}$) of the rank 1 oriented matroid $M|_{F_p}/F_{p-1}$.
\end{proof}

\begin{proposition} \label{prop:gamma_epsilon}
    The map $\tilde{a}_p:\Pcal_p(M) \rightarrow\Acal_p(M)$ given by extending $\gamma_{\Fcal,\vv,p}\mapsto \epsilon_{\Fcal,\vv,p}$ linearly is well-defined, surjective, and has kernel equal to $\Pcal_{p+1}(M)$. In particular, $\tilde{a}_p$ descends to an isomorphism $\Pcal_p(M)/\Pcal_{p+1}(M)\cong \Acal_p(M)$. 
\end{proposition}

\begin{proof}
    We show that the map $\Pcal_p(M)/\Pcal_{p+1}(M)\rightarrow\Acal_p(M)$ given by extending $[\gamma_{\Fcal,\vv,p}]\mapsto \epsilon_{\Fcal,\vv,p}$ linearly is well-defined and is an isomorphism. The statement in the proposition follows from composing the quotient map $\Pcal_p(M)\rightarrow\Pcal_p(M)/\Pcal_{p+1}(M)$ with this isomorphism.

    We claim that the map $\Pcal_p(M)/\Pcal_{p+1}(M)\rightarrow\Acal_p(M)$ is the composition of three isomorphisms: 
    \begin{enumerate}
\item  $\Pcal_p(M)/\Pcal_{p+1}(M)\cong \Hom(\Pcal^p(M)/\Pcal^{p-1}(M), \ZZ)$, induced from the perfect pairing $\Pcal^p(M)/\Pcal^{p-1}(M)\times \Pcal_p(M)/\Pcal_{p+1}(M)\rightarrow\ZZ$; 
\item $\Hom(\Pcal^p(M)/\Pcal^{p-1}(M), \ZZ)\cong \Hom(\Acal^p(M), \ZZ)$, induced by the pullback of the isomorphism $\Acal^p(M)\cong \Pcal^p(M)/\Pcal^{p-1}(M)$ given by $(x^*)^S\mapsto [\prod_{i\in S} h_i]$ \cite[Theorem~5.9]{Mose17}; 
\item $\Acal_p(M)\cong \Hom(\Acal^p(M), \ZZ)$, induced by the perfect pairing $\Acal^p(M)\times \Acal_p(M)\rightarrow\ZZ$.
    \end{enumerate}

    Since $\Pcal^p(M)/\Pcal^{p-1}(M)$ is generated by products of Heaviside functions $h_{i_1}\ldots h_{i_p}$ and $\Pcal_p(M)/\Pcal_{p+1}(M)$ is generated by prefix chains $\gamma_{\Fcal,\vv,p}$, it suffices to verify the evaluation of $h_{i_1}\ldots h_{i_p}$ on $\gamma_{\Fcal,\vv,p}$ is equal to that of $(x^*)^{\{i_1,\ldots,i_p\}}$ on $\epsilon_{\Fcal,\vv,p}$. We first assume $i_j\in F_j\setminus F_{j-1}$ for every $j$ (up to reindexing). On the Varchenko--Gelfand side, if $T'$ is the unique tope in $U_{\Fcal,\vv,p}$ such that $T'(i_j)=+$ for every $j$, then the evaluation is equal to the sign of $T'$ in $\gamma_{\Fcal,T,p}$. This is $-1$ to the power of the number of $j$'s such that $T'(i_j)$ differ from $T(i_j)$, equivalently $T(i_1)\ldots T(i_p)$.
    The evaluation on the Cordovil side is also equal to $T(i_1)\ldots T(i_p)$.

    Suppose $\{i_1,\ldots,i_p\}\cap (F_j\setminus F_{j-1})=\emptyset$ for some $j\leq p$ instead. On the Varchenko--Gelfand side, we pair up the elements of $U_{\Fcal,\vv,p}$ into $\{\uu_i,\uu_i+\dd_j\}$'s, whose corresponding terms have opposite sign upon being evaluated by $h_{i_1}\ldots h_{i_p}$, thus the overall evaluation is zero. The evaluation is also zero on the Cordovil side by the simple observation about the monomials in the support of $\epsilon_{\Fcal,\vv,p}$.
\end{proof}

Using the above calculation on evaluating prefix chains by monomials of Heaviside functions, together with the same reasoning as Proposition~\ref{prop:alpha_dual_basis} and Corollary~\ref{coro:NBC_Hp_basis}, we have the following corollary.

\begin{corollary} \label{coro:NBC_Pp_basis}
    For every NBC-set $S$ of size $p$, extend the NBC-chain of $S$ arbitrarily into a complete flag $\tilde{\Fcal}_S$ and pick an $T_S\in\Tcal_{\tilde{\Fcal}_S}$.
    Then the collection of (the equivalent classes of) prefix chains $\gamma_{\tilde{\Fcal}_S,T_S,p}$ form a $\ZZ$-basis of $\Pcal_p(M)/\Pcal_{p+1}(M)$. In fact, it is the dual basis (up to negation) of $\big\{\prod_{i\in S} h_i: S\in\BC^p(\underline{M})\big\}\subset \Pcal^p(M)/\Pcal^{p-1}(M)$.
\end{corollary}

The simple description of $\tilde{a}_p$ goes beyond prefix chains.

\begin{proposition} \label{prop:affine_image}
    Let $\Fcal$ be a complete flag and $U=\vv+\langle \dd_{i_1},\ldots,\dd_{i_p}\rangle$ be an affine coordinate subspace of the tope space $\Tcal_\Fcal$. Then $$\tilde{a}_p(\gamma_{U,\Bcal,\vv})=\prod_{j=1}^p \big(\sum_{k\in F_{i_j}\setminus F_{i_j-1}} T_\vv(k) x_k\big).$$
\end{proposition}

\begin{proof}
    The argument is similar to that of Proposition~\ref{prop:gamma_epsilon}, namely, by comparing the evaluation of $\gamma_{U,\Bcal,\vv}$ by $h_{i'_1}\ldots h_{i'_p}$'s and the evaluation of the right hand side by $(x^*)^{\{i'_1,\ldots,i'_p\}}$'s. The calculation is also essentially the same as in the proof of Proposition \ref{prop:gamma_epsilon}. 
\end{proof}

The next lemma helps us to relate flags of flats of an initial matroid with flags of flats of the original matroid. It is used to define maps between the Cordovil algebras of initial oriented matroids to that of the original oriented matroid.

\begin{lemma} \label{lem:IM_flag}
    Let $\Fcal:\emptyset=F_0 \subsetneq  F_1 \subsetneq \ldots \subsetneq F_l=E$ be a flag of flats (not necessarily complete nor with $\rank F_i=i$). Let $G_0 \subsetneq  \ldots \subsetneq  G_d$ be a complete flag of the initial matroid $\underline{M}_\Fcal$. Then there exists a complete flag $G'_0 \subsetneq  \ldots \subsetneq G'_d$ of $\underline{M}$ such that $\{G_i\setminus G_{i-1}: 1\leq i\leq d\}=\{G'_i\setminus G'_{i-1}: 1\leq i\leq d\}$.
\end{lemma}

\begin{proof}
    Every flat $G$ of $\underline{M}_\Fcal$ can be written uniquely as $G^{(1)}\sqcup\ldots\sqcup G^{(l)}$, where $G^{(i)}\sqcup F_{i-1}\subset F_i$ is a flat of $\underline{M}$ for each $i$. So for any two consecutive flats $G_j\subsetneq G_{j+1}$, there must exist a unique index $1\leq k\leq l$ such that $G_j^{(t)}=G_{j+1}^{(t)}$ for any $t\neq k$ and $G_j^{(k)}\subsetneq G_{j+1}^{(k)}$ whose ranks differ by 1. Now the chain $G_0^{(1)}\leq G_1^{(1)}\leq\ldots\leq G_d^{(1)}\leq G_0^{(2)}\cup F_1\leq G_1^{(2)}\cup F_1\leq\ldots\leq G_d^{(2)}\cup F_1\leq \ldots\leq G_d^{(l)}\cup F_{l-1}$ has exactly $d+1$ distinct elements and form a complete flag of $\underline{M}$, which satisfies the requirement.
\end{proof}

 \begin{proposition}\label{prop:tope_filtration_incl}
For every flag $\Fcal$ there are canonical inclusion maps 
$\iota: \ZZ[\Tcal_{\Fcal}] \to  \ZZ[\Tcal(M)]$ and $\iota:\Pcal_p(M_\Fcal)\hookrightarrow \Pcal_p(M)$ for all $p$.
 \end{proposition}

 \begin{proof}
By Definition \ref{def:tope_flag}, we have $\Tcal_{\Fcal} \subset \Tcal(M)$ which gives an immediate inclusion map for the tope spaces. 

Every function in $\ZZ[\Tcal]^*$ can be restricted as a function in $\ZZ[\Tcal_\Fcal]^*$, and the restriction of the Heaviside function $h_e\in\ZZ[\Tcal]^*$ is $h_e\in \ZZ[\Tcal_\Fcal]^*$. Hence, when the function is written as a polynomial in Heaviside functions, the same polynomial (in Heaviside functions over $\ZZ[\Tcal_\Fcal]$) represents its restriction in $\ZZ[\Tcal_\Fcal]^*$.
Conversely, every function in $\ZZ[\Tcal_\Fcal]^*$ can be written in a polynomial in Heaviside functions (over $\ZZ[\Tcal_\Fcal]$), in which the presentation defines a function in $\ZZ[\Tcal]^*$, whose restriction is the function we started with.

By choosing the presentation of the lowest possible degree in the above conversion, we have that $\Pcal^p(M_\Fcal)$ consists of precisely the restriction of the functions in $\Pcal^p(M)$. Now $\gamma$ is in $\Pcal_p(M_\Fcal)$ if and only if it is evaluated to 0 by every function in $\Pcal^{p-1}(M_\Fcal)$, so $\gamma$ as an element in $\ZZ[\Tcal(M)]$ also gets evaluated to 0 by every function in $\Pcal^{p-1}(M)$ thus is in $\Pcal_p(M)$.
\end{proof}
 
\begin{proposition} \label{prop:Cordovil_incl}
    There is a canonical inclusion map $\iota^{\Acal}:\Acal_p(M_\Fcal)\hookrightarrow \Acal_p(M)$ for every flag $\Fcal$.
\end{proposition}

\begin{proof}
Both $\Acal_p(M_{\Fcal})$ and $\Acal_p(M)$ live in ${\rm SFSym}(E)$ as submodules.
   The Cordovil algebra $\Acal_p(M_\Fcal)$ is generated by $\epsilon_{\Gcal,\vv,p}$'s where $\Gcal$ is a complete flag of $\underline{M}_\Fcal$. Such an element is the image of the prefix chain $\gamma_{\Gcal,\vv,p}\in\Pcal_p(M_\Fcal)$ under $\tilde{a}_p$. By Lemma~\ref{lem:IM_flag}, there exists a complete flag $\Gcal'$ of $\underline{M}$ such that $\gamma_{\Gcal,\vv,p}$ is an affine coordinate chain in $\Pcal_p(M)$ with respect to it. By Proposition~\ref{prop:affine_image}, the image of such an affine chain under $\tilde{a}_p$, which is necessarily in $\Acal_p(M)$, is equal to $\epsilon_{\Gcal,\vv,p}$ (the support of the factors are the same $G_i\setminus G_{i-1}$'s, and their sign patterns are inherited from the same $T_\vv$).
\end{proof}
\subsection{Proof of Theorem~\ref{mainthm:Z_cosheaf}} \label{sec:ThmC}

We begin by defining the cosheaves on the fan of a matroid that are involved in Theorem~\ref{mainthm:Z_cosheaf}. 
Recall that the fan of a matroid $\Sigma_{\underline{M}}$ has cones corresponding to the flags of flats of the matroid, hence to its initial matroids (Definition \ref{def:matfan}). 
For flags of flats $\mathcal{F}, \mathcal{F}'$, there is an inclusion of cones $\sigma_{\mathcal{F}} \subset \sigma_{\mathcal{F}} $ if and only if $\mathcal{F}' \subset \mathcal{F}$. 

Recall that we view $\Sigma_{\underline{M}}$ as a category with objects corresponding to the  faces of the fan  and morphisms corresponding to inclusions of faces. From Definition \ref{def:sheaf}, a cosheaf on $\Sigma_{\underline{M}}$ is a  functor 
$\mathfrak{G}: \Sigma_{\underline{M}}^{\text{op}} \to  {\rm Mod}_{\ZZ}$, where $ {\rm Mod}_{\ZZ}$ is the category of $\ZZ$-modules and $\Sigma_{\underline{M}}^{\text{op}}$  has morphisms in the opposite direction from $\Sigma_{\underline{M}}.$ 

We construct a collection of cosheaves on $\Sigma_{\underline{M}}$ by showing that the dual Varchenko--Gelfand filtration of the $\ZZ$-tope space of $M$, as well as Cordovil algebras, are functorial with respect to initial matroids. We begin with defining a $\ZZ$-variant of the sign cosheaf from \cite{RS}. 

\begin{definition}\label{def:signcosheaf}
The $\ZZ$-sign cosheaf $\Scal$ of an oriented matroid is defined by the assignment 
$\Scal(\sigma_{\mathcal{F}}) = \ZZ[\Tcal(M_\Fcal)]$ for each cone $\sigma_{\Fcal}$ of $\Sigma_{\underline{M}}$ and inclusion maps $\Scal(\sigma_{\mathcal{F}}) \to \Scal(\sigma_{\mathcal{F}'})$ for each pair of flags $\mathcal{F}', \mathcal{F}$ with $\mathcal{F}' \leq \mathcal{F}$.

The $\Ztwo$-variant simply assigns the corresponding $\Ztwo$-tope space to each face again with morphisms being inclusion maps. 
\end{definition}

That there are inclusion maps of the tope spaces of $M_\mathcal{F}$ and $M_{\mathcal{F}'}$ when $\sigma_{\mathcal{F}'} \subset \sigma_{\mathcal{F}}$ follows from Proposition \ref{prop:tope_filtration_incl}.
 
 The $\Ztwo$-sign cosheaf was filtered using Quillen filtration from Section \ref{sec:Quillen} in \cite{RS, RRS}.
Using the dual Varchenko--Gelfand filtration, we can extend this filtration to the integral sign cosheaf.

\begin{proposition}\label{prop:GVcosheaf}
    The map $\Pfrak_p$ that takes $\sigma_\Fcal$ to $\Pcal_p(M_\Fcal)$ and containment of cones $\sigma_{\mathcal{F}'} \subseteq \sigma_{\mathcal{F}}$ to inclusion of $\ZZ$-modules $\Pcal_p(M_\Fcal)\rightarrow \Pcal_p(M_{\Fcal'})$ provided by Proposition \ref{prop:tope_filtration_incl} is a cosheaf on the fan of a matroid.
\end{proposition}

\begin{proof}
    It must be shown that the induced inclusion maps between the $\ZZ$-tope spaces for $\Fcal \leq \Fcal'$ respect compositions. This holds since if there are two flags satisfying $\Fcal \leq \Fcal_1,\Fcal_2 \leq \Fcal'$ then the following diagram commutes: 

   \begin{equation} \label{eq:VGSheaf}
\begin{CD}
\Pcal_p(M_{\Fcal})     @>\iota>>  \Pcal_p(M_{\Fcal_1}) \\
@V\iota VV        @V\iota VV\\
\Pcal_p(M_{\Fcal_2})    @>\iota>> \Pcal_p(M_{\Fcal'}),
\end{CD}
\end{equation}
and $\mathfrak{P}_p$ is a functor. 
\end{proof}

\begin{corollary}\label{cor:cosheafFiltration}
The dual Varchenko--Gelfand filtration provides a filtration of the integral sign cosheaf of an oriented matroid. Namely, there are inclusions of cosheaves, 
$$\Pfrak_d  \subset  \dots \subset \Pfrak_p \subset \Pfrak_{p-1} \subset \dots \subset \Pfrak_1 \subset \mathcal{S},$$ where the inclusions denote the existence of injective cosheaf maps. 
\end{corollary}

\begin{proposition}\label{prop:CordovilCosheaf}
    The map $\mathfrak{A}_p$ that takes $\sigma_\Fcal$ to $\Acal_p(M_\Fcal)$ and containment of cones $\sigma_{\mathcal{F}'} \subseteq \sigma_{\mathcal{F}}$ to inclusion of $\ZZ$-modules $\Acal_p(M_\Fcal)\rightarrow \Acal_p(M_{\Fcal'})$ provided by Proposition \ref{prop:Cordovil_incl} is a cosheaf on the fan of a matroid.
\end{proposition}

\begin{proof}
    Similar to the case of $\Pfrak_p$, since inclusion maps commute, for flags $\Fcal \leq \Fcal_1,\Fcal_2 \leq \Fcal'$ we have the commutative diagram:
   \begin{equation} \label{eq:CordSheaf}
    \begin{CD}
        \Acal_p(M_{\Fcal})     @>\iota^\Acal>>  \Acal_p(M_{\Fcal_1}) \\
        @V\iota^\Acal VV        @V\iota^\Acal VV\\
        \Acal_p(M_{\Fcal_2})    @>\iota^\Acal>> \Acal_p(M_{\Fcal'}).
    \end{CD}
    \end{equation}
    This completes the proof. 
\end{proof}

We call $\mathfrak{A}_p$ the {\em $p$-th Cordovil cosheaf}.
Using the previous sections, which identified the intermediate quotients of dual Varchenko--Gelfand filtration with the dual of the Cordovil algebra, we can find short exact sequences of cosheaves on $\Sigma_{\underline{M}}$ involving the cosheaves of the filtrations and the Cordovil cosheaves. We do so by first establishing the next lemma.

\begin{lemma} \label{lem:CD_for_MF}
Let $\Fcal'$ and $\Fcal$ be flags of flats of $\underline{M}$ with $\Fcal' \leq \Fcal$.
We have the following commutative diagram: 
\begin{equation} \label{eq:SES}
\begin{CD}
0 @> >> \Pcal_{p+1}(M_{{\Fcal}})     @>\iota >> \Pcal_p(M_{{\Fcal}})     @>\tilde{a}_p>> \Acal_p({M_{\Fcal}}) @> >> 0 \\
@.  @V\iota VV  @V\iota VV        @V\iota^{\Acal} VV\\
 0 @> >>  \Pcal_{p+1}({M_{{\Fcal'}}})     @>\iota >> \Pcal_p(M_{{\Fcal'}})     @>\tilde{a}_p>>   \Acal_p(M_{{\Fcal'}})@> >> 0 
\end{CD},
\end{equation}
where $\iota$'s are inclusions within the tope space $\ZZ[\Tcal]$, $\iota^{\Acal}$ is an inclusion within ${\rm SFSym}(E)$, and $\tilde{a}_p$'s are the respective $\tilde{a}_p$ maps for the two oriented matroids.
\end{lemma}

\begin{proof}
The horizontal inclusion maps $\iota:\Pcal_{p+1}(M)\rightarrow \Pcal_p(M)$ come from the respective filtrations. By Propositions \ref{prop:tope_filtration_incl} and \ref{prop:Cordovil_incl}, the vertical arrows are all inclusion maps. 
Therefore, the leftmost square consisting of inclusion maps commutes trivially.

Since $\Pcal_p(M_\Fcal)$ is generated by prefix chains $\gamma_{\Gcal,\vv,p}$'s of $M_\Fcal$, it suffices to show the rightmost square commutes for these chains. But this is essentially the proof of Proposition~\ref{prop:Cordovil_incl}: the image of $\gamma_{\Gcal,\vv,p}$ under the $\tilde{a}_p$ map of $M_\Fcal$ is $\epsilon_{\Gcal,\vv,p}$, which is sent to the same element in $\Acal_p(M)$ via $\iota^\Acal$, such an element is also the image of the $\gamma_{\Gcal,\vv,p}$, viewed as an affine coordinate chain of $\Pcal_p(M)$, under the $\tilde{a}_p$ map of $M$.
\end{proof}

\begin{corollary} \label{coro:SES_cosheaves}
For an oriented matroid $M$ there is an exact sequence of cosheaves 
$$0 \to \Pfrak_{p+1} \to \Pfrak_p \to \mathfrak{A}_p \to 0$$
for every $p$. 
\end{corollary}

\begin{proof}
   The short exact sequence of cosheaves follows from Lemma \ref{lem:CD_for_MF} and Propositions \ref{prop:GVcosheaf} 
and \ref{prop:CordovilCosheaf}. 
\end{proof}

\begin{proof}[Proof of Theorem \ref{mainthm:Z_cosheaf}]
The proof now follows from Corollaries \ref{cor:cosheafFiltration} and \ref{coro:SES_cosheaves}. 
\end{proof}

\begin{remark} \rm
    We mention some connection between our construction and operad theory as communicated to us by Basile Coron. In the case of the braid arrangement $\{\{x_i=x_j\}:1\leq i\neq j\leq n\}$, the $\ZZ$-tope space can be viewed as the group algebra $\ZZ[\mathfrak{S}_n]$ of the symmetric group, and the {\em associative operad} that relates symmetric groups of different sizes is encoded by the cosheaf structure of $\mathcal{S}$. Under such an identification, prefix chains correspond to (composition of) Poisson brackets, and the cosheaf structure of the Cordovil cosheaves encode the {\em Poisson operad}. Such a connection can be extended to more general (oriented) matroids. We refer the reader to \cite{Coron2025} and the references therein for details on poset operad theory.
\end{remark}

In the original short exact sequence over $\Ztwo$ in \cite{RS}, the Cordovil cosheaf is replaced with  the mod 2 reduction of the {\em $\ZZ$-tropical homology cosheaf} on $\Sigma_{\underline{M}}$ (also known as the $p$-th multi-tangent space) \cite{IKMZ}, which  assigns $\OS_p(\underline{M}_{\Fcal}; \Ztwo)$ to $\sigma_{\Fcal}$ and with cosheaf maps the inclusion maps $\iota: \OS_p(\underline{M}_{\Fcal}; \Ztwo) \to \OS_p(\underline{M}_{\Fcal'}; \Ztwo)$. The mod 2 reduction of the $\mathcal{A}_p$ is the same as the mod 2 reduction of $\OS_p$, and thus the mod 2 reductions of these cosheaves are equal. 

The next example shows that a straight up variant of Corollary \ref{coro:SES_cosheaves} with $\mathfrak{A}_p$ replaced with the cosheaf assigning $\OS_p(\underline{M}_{\Fcal}; \ZZ)$ to all faces and inclusion maps is impossible. More precisely, it is not possible to lift the original short exact sequence in \cite{RS} to $\ZZ$-coefficients in a way that replaces $\Acal_p$ with $\OS_p(\  \cdot \ ; \ZZ)$ in the commutative diagram (\ref{eq:SES})
while maintaining that all of the cosheaf maps as inclusion maps.

\begin{example}
\rm
Consider the oriented matroid $M$ in Figure~\ref{diag:U23}, and its initial matroids corresponding to the three flags $\Fcal_i:\emptyset \subsetneq \{L_i\} \subsetneq  E$. We show that there does not exist a compatible system of $\abv_1^{(i)}:\Pcal_1(M_{\Fcal_i};\ZZ)\rightarrow\OS_1(\underline{M}_{\Fcal_i};\ZZ)$ such that (1) the maps lift the original $\qbv_1$ maps in \cite{RS}, (2) the commutative diagram (\ref{eq:SES}) holds for every $\Fcal_i$ including the vertical inclusion maps (with $\Acal$'s replaced by $\OS$'s in the obvious way). In particular, since $\OS_1(\underline{M}_{\Fcal_i})$ is generated by $\ee_i$ and $\ee_{E\setminus\{i\}}$ by Proposition~\ref{prop:OSZhar}, each prefix chain of $M_{\Fcal_i}$ is sent to $c\ee_i+c'\ee_{E\setminus\{i\}}$ where the parity of $c,c'\in\ZZ$ depends on whether the prefix chain is with respect to $\emptyset\subsetneq \{L_i\}$ or $\emptyset \subsetneq E\setminus\{L_i\}$. In the calculation below, $\alpha$'s stand for odd integers whereas $\beta$'s are even.

Suppose $\abv_1^{(1)}$ sends $T_1-T_2$ to $\alpha\ee_1+\beta\ee_2+\beta\ee_3$ and $\abv_1^{(2)}$ send $T_2-T_3$ to $\beta'\ee_1+\alpha'\ee_2+\beta'\ee_3$. Then $\abv_1^{(3)}$ sends $T_1-T_3=(T_1-T_2)+(T_2-T_3)$ to $(\alpha+\beta')\ee_1+(\alpha'+\beta)\ee_2+(\beta+\beta')\ee_3$. Now $\abv_1^{(3)}$ must send $T_3-T_4$ to $(\alpha''+\beta'-\alpha')\ee_1+(\alpha''+\beta'-\alpha')\ee_2+\alpha''\ee_3$ for some $\alpha''$ just so $\abv_1^{(1)}$ can send $T_2-T_4=(T_2-T_3)+(T_3-T_4)$ to $(\alpha''+2\beta'-\alpha')\ee_1+(\alpha''+\beta')\ee_2+(\alpha''+\beta)\ee_3$ (here we use the constraint that $\ee_2,\ee_3$ must have the same coefficients). Next, since $(T_1-T_2)+(T_4-T_5)$ is in $\Pcal_2(M_{\Fcal_1})$, $\abv_1^{(1)}$ must send it to $0$, and $\abv_1^{(1)}(T_4-T_5)=-\alpha\ee_1-\beta\ee_2-\beta\ee_3$. So $\abv_1^{(2)}$ sends $T_3-T_5=(T_3-T_4)+(T_4-T_5)$ to $(\alpha''+\beta'-\alpha-\alpha')\ee_1+(\alpha''+\beta'-\alpha'-\beta)\ee_2+(\alpha''-\beta)\ee_3$.

Now comparing the coefficients of $\ee_1,\ee_2$ in the image of $T_1-T_3$ yields $\beta-\beta'=\alpha-\alpha''$, whereas comparing the coefficients of $\ee_1,\ee_3$ in the image of $T_3 - T_5$ yields $\beta+\beta'=\alpha+\alpha'$, these imply $\alpha=\beta,\alpha'=\beta'$, a contradiction as they should have different parity.
\end{example}

\subsection{Projectivisation of the Varchenko--Gelfand Filtration}

In \cite{RRS2, RS}, the authors consider not just the $\Ztwo$-tope space of an oriented matroid but also the {\em projective $\Ztwo$-tope space}. The intermediate gradeds of the filtration of the projective $\Ztwo$-tope space are shown to be isomorphic to the stalks of the cosheaves from $\Ztwo$-tropical homology $\mathcal{F}_p(M)$. 
By \cite{Zhar_OS}, this vector space is also isomorphic to the dual of a graded piece of the {\em projective Orlik--Solomon algebra} \cite{OS,Kawahara}, which we denote by $\overline{\OS}_p(\underline{M}; \Ztwo)$.

Here we consider a $\ZZ$-variant of this projectivisation of the tope space and also the dual Cordovil algebra. Consider the collection $\Tcal^\pj$ of projective topes, obtained by identifying antipodal topes in $\Tcal$. We have the canonical projection $\pi:\ZZ[\Tcal]\rightarrow\ZZ[\Tcal^\pj]$; denote by $\Theta$ the kernel of this map. 
We show that the dual Varchenko--Gelfand filtration induces a filtration on $\ZZ[\Tcal^\pj]$. In what follows we shorten 
$\Pcal_p(M)$ and $\Acal_p(M)$ to $\Pcal_p$ and $\Acal_p$, respectively, since we only consider one oriented matroid $M$. 

Consider the following diagram

$$
\begin{CD}
@. 0 @. 0 @. 0\\
@. @VVV @VVV @VVV\\
0@>>>\Theta\cap\Pcal_{p+1}@>>>\Pcal_{p+1}@>>>\Pcal_{p+1}^\pj@>>>0\\
@. @VVV @VVV @VVV\\
0@>>>\Theta\cap\Pcal_p@>>>\Pcal_p @>>>\Pcal_p^\pj@>>>0\\
@. @VV{\tilde{a}_p|_\Theta} V @VV{\tilde{a}_p}V @VV{\tilde{a}_p^\pj}V\\
0@>>> B_p @>>>\Acal_p@>>>\Acal_p^\pj@>>>0\\
@. @VVV @VVV @VVV\\
@. 0 @. 0 @. 0.\\
\end{CD}
$$
Here $\Pcal_{p+1}^\pj=\pi(\Pcal_{p+1}), \Pcal_p^\pj=\pi(\Pcal_p), B_p:=\tilde{a}_p(\Theta\cap\Pcal_p)$, and $\Acal_p^\pj:=\Acal_p/B_p$ ({\em a priori} not $\Pcal_p^\pj/\Pcal_{p+1}^\pj$).

\begin{proposition}
    The diagram is well-defined (the map $\tilde{a}_p^\pj$ exists), and it is commutative and exact.
\end{proposition}

\begin{proof}
The top-left square commutes since every map is an inclusion map. The commutativity of the top-right square is because the two horizontal maps are both restrictions of $\pi$ whereas the two vertical maps are inclusions; analogous reasoning for the lower-left square.
The existence of $\tilde{a}_p^\pj$ and the commutativity of the lower-right square follow from the snake lemma applied to the first two columns once we have proved the exactness of the first two columns.

The exactness of the first two rows follows the first isomorphism theorem, and that of the third row from definition. The middle column is Lemma~\ref{lem:CD_for_MF}. 
Next we prove the exactness of the first column, i.e., the kernel of $\tilde{a}_p$ restricted to $\Theta\cap\Pcal_p$ is $\Theta\cap\Pcal_{p+1}$. The L.H.S. is contained in $\Theta$ by definition, as well as in the kernel of $\tilde{a}_p$ restricted to $\Pcal_p$, which is $\Pcal_{p+1}$; conversely, $\Theta\cap\Pcal_{p+1}\subset\Theta\cap\Pcal_p$, while $\tilde{a}_p(\Theta\cap\Pcal_{p+1})\subset \tilde{a}_p(\Pcal_{p+1})=\{0\}$. The exactness of the last column follows from the nine lemma.
\end{proof}

\begin{theorem}
    When $p$ is even, $\Acal_p^\pj\cong \Acal_p$. When $p$ is odd, $\Acal_p^\pj\cong\overline{\OS}_p(\underline{M}; \Ztwo)$. 
\end{theorem}

\begin{proof}
    For $p$ even, given $\gamma=\sum \gamma_{\Fcal,\vv,p}\in\Theta\cap\Pcal_p$ (possibly with repeated summands), we have $2\gamma=\sum [\gamma_{\Fcal,\vv,p}-\gamma_{\Fcal,-\vv,p}]$ by the anti-symmetry of elements in $\Theta$. Hence $2\cdot \tilde{a}_p(\gamma)=\sum [\tilde{a}_p(\gamma_{\Fcal,\vv,p})-\tilde{a}_p(\gamma_{\Fcal,-\vv,p})]=0$, but as $B_p$ is a subgroup of the free abelian group $\Acal_p$, $\tilde{a}_p(\gamma)=0$. In particular, $B_p=\{0\}$ and $\Acal_p^\pj = \Acal_p/B_p\cong \Acal_p$.

    For $p$ odd (which is the assumption for the rest of this proof), we also have that $\gamma_{\Fcal,\vv,p}-\gamma_{\Fcal,-\vv,p}\in\Theta\cap\Pcal_p$ for any prefix chain $\gamma_{\Fcal,\vv,p}\in\Pcal$, except that $\tilde{a}_p(\gamma_{\Fcal,\vv,p}-\gamma_{\Fcal,-\vv,p})=2\cdot \tilde{a}_p(\gamma_{\Fcal,\vv,p})\in B_p$ for parity reason in the definition of $\tilde{a}_p$. Since the prefix chains generate $\Pcal_p$, $2\cdot \Acal_p\subset B_p$ and $\Acal_p^\pj$ is a quotient of $\Acal_p/2\Acal_p$, in particular a vector space over $\Ztwo$.
  
    The dimension of $\Acal_p^\pj$ over $\Ztwo$ is at least that of $\overline{\OS}_p(\underline{M}; \Ztwo)$: tensoring the short exact sequence $0\rightarrow \Pcal_{p+1}^\pj\rightarrow \Pcal_p^\pj\rightarrow \Acal_p^\pj\rightarrow 0$ with $\Ztwo$ gives $\overline{\Pcal_{p+1}^\pj}\rightarrow \overline{\Pcal_p^\pj} \rightarrow \Acal_p^\pj\rightarrow 0$. Meanwhile, by combining Theorem \ref{mainthm:maps} with \cite[Proposition 5.7]{RRS2}, we find that $\overline{\Pcal_p^\pj}/\overline{\Pcal_{p+1}^\pj} \cong \overline{\OS}_p(\underline{M}; \Ztwo)$.

    To see that this is exactly the dimension of $\Acal_p^{\pj}$, fix an arbitrary ordering of $E$, we claim that the projection of NBC-chains indexed by NBC-sets that do not contain the minimum element $i_1\in E$ are sufficient to span $\Pcal_p^\pj/\Pcal_{p+1}^\pj$. Since the number of such NBC-sets of size $p$ is the dimension of $\overline{\OS}_p(\underline{M}; \Ztwo)$, the tensor product above is actually exact. Let $\{i_1,i_2,\ldots,i_p\}$ be a NBC-set ($i_2,\ldots,i_p$ are not necessarily the next $p-1$ smallest elements) and $\Fcal$ be the respective NBC-flag. By reorientation if necessary, without loss of generality we consider $\gamma_{\Fcal,{\bf 0},p}=\gamma_{\Fcal,{\bf 0},p-1}-\gamma_{\Fcal,\dd_p,p-1}$, whose image in $\ZZ[\Tcal^\pj]$ is equal to that of $\gamma':=\gamma_{\Fcal,{\bf 0},p-1}-\gamma_{\Fcal,{\bf 1}-\ff_p,p-1}$.
    We have $\gamma'\in\Pcal_p$: let $h_{j_1}\ldots h_{j_{p-1}}\in\Pcal^{p-1}$, we may assume $j_k\in F_k\setminus F_{k-1}$ for every $k$, or otherwise $h_{j_1}\ldots h_{j_{p-1}}(\gamma_{\Fcal,{\bf 0},p-1})=h_{j_1}\ldots h_{j_{p-1}}(\gamma_{\Fcal,{\bf 1}-\ff_p,p-1})=0$; now each tope $T_\vv$ in the support of $\gamma_{\Fcal,{\bf 0},p-1}$ can be paired up with $T_{\vv+{\bf 1}-\ff_p}$ of $\gamma_{\Fcal,{\bf 1}-\ff_p,p-1}$, which have the same evaluation under $h_{j_1}\ldots h_{j_{p-1}}$ and of same coefficients in the respective chains, hence $h_{j_1}\ldots h_{j_{p-1}}(\gamma')=0$ as well. Therefore, we may write $\gamma'$ uniquely as a linear combination of NBC-chains, which we claim those indexed by NBC-sets that contains $i_1$ do not appear: let $S\ni i_1$ be such a NBC-set, consider $h^S:=\prod_{i\in S} h_i$, the same tope pairing argument shows that $h^S(\gamma')=0$, which concludes the argument by Corollary~\ref{coro:NBC_Pp_basis}.4
\end{proof}

\bibliographystyle{plain}
\bibliography{FF}

\medskip

Department of Mathematics, University of Oslo, Norway

Email address: \url{krisshaw@math.uio.no}\\

Department of Applied Mathematics, National Yang Ming Chiao Tung University, Taiwan

Email address: \url{chyuen@math.nctu.edu.tw}\\

\end{document}